\documentclass[11pt,a4paper,onesided]{article}

\usepackage{ amssymb, amsmath, amsthm, graphics, xspace, enumerate}
\usepackage{graphicx}
\usepackage[a4paper,colorlinks=true,citecolor=blue,urlcolor=blue,linkcolor=blue,bookmarksopen=true,unicode=true,pdffitwindow=true]{hyperref}
 \hypersetup{pdfauthor={Nikola Sandri\'{c}}}
\hypersetup{pdftitle={Homogenization of
periodic diffusion with small jumps}}

\newtheorem{tm}{tm}[section]
\newtheorem{theorem}[tm]{Theorem}

\newtheorem{corollary}[tm]{Corollary}
\newtheorem{proposition}[tm]{Proposition}

\newtheorem{definition}[tm]{Definition}

\newtheorem{remark}[tm]{Remark}

\newcommand{\process}[1]{\{#1_t\}_{t\geq0}}

\newcommand {\R} {\ensuremath{\mathbb{R}}}
\newcommand {\ZZ} {\ensuremath{\mathbb{Z}}}
\newcommand {\N} {\ensuremath{\mathbb{N}}}
\newcommand {\CC} {\ensuremath{\mathbb{C}}}

\numberwithin{equation}{section}
\def\be{\begin{equation}}
\def\ee{\end{equation}}
\begin{document}

 \title{Homogenization of
periodic diffusion with small jumps}
  \author{
Nikola Sandri\'{c}\\ Institut f\"ur Mathematische Stochastik\\ Fachrichtung Mathematik, Technische Universit\"at Dresden, 01062 Dresden, Germany\\
and\\
Department of Mathematics\\
         Faculty of Civil Engineering, University of Zagreb, 10000 Zagreb,
         Croatia \\
        Email: nsandric@grad.hr}

 \maketitle
\begin{center}
{
\medskip

} \end{center}

\begin{abstract}
In this paper, we study the homogenization of  a diffusion process with jumps, that is,   Feller process generated by an
 integro-differential operator.  This problem is closely related to the problem of  homogenization of boundary value problems arising in studying the behavior of heterogeneous media.
 Under the assumptions that the corresponding generator  has vanishing drift coefficient,    rapidly periodically oscillating  diffusion and jump coefficients,  that it admits only ``small  jumps" (that is, the jump kernel has  finite second moment) and under certain additional regularity conditions, we prove that the homogenized process is a Brownian motion.
 The presented results generalize the  classical and well-known results related to
  the homogenization of a  diffusion process.

\end{abstract}
{\small \textbf{AMS 2010 Mathematics Subject Classification:} 35S15, 47G20, 60F17,  60J75}
\smallskip

\noindent {\small \textbf{Keywords and phrases:} diffusion with jumps,  homogenization,   semimartingale characteristics}

%
%
%
%


\section{Introduction}\label{s1}
Many phenomena arising in nature, engineering and social sciences involve heterogeneous media, such as  problems related to diffusion of population, composite
materials and  large financial market
movements. Because
of heterogeneity,
mathematical models  used in describing these phenomena (typically
stochastic processes or integro-differential equations) are characterized by heterogeneous coefficients, and as such are very complicated and hard to analyze. However, on macroscopic scales, they often show an effective behavior which is a result of macroscopic
scale averaging of the complicated microscopic scale structure.
More precisely, in many cases
when the coefficients
 (rapidly)
vary on  small scales
it is possible to use the fine microscopic structure of the media
 to derive an effective (homogenized) model which is a valid approximation of the
initial model and, in general, it is of much simpler from (typically it is characterized by
constant coefficients).

The  goal of this paper is to investigate the homogenization of a $d$-dimensional Feller process
$\process{F^{\varepsilon}}$    generated by a non-local (integro-differential) operator (with vanishing drift term) of the form
\begin{equation}\label{eq1.1}\mathcal{A}_{\varepsilon} f(x):=
\frac{1}{2}\sum_{i,j=1}^{d}c_{ij}(x/\varepsilon)\frac{\partial^{2}f(x)}{\partial x_i\partial x_j}
+\frac{1}{\varepsilon^{2}}\int_{\R^{d}}\left(f(\varepsilon y+x)-f(x)-\varepsilon\sum_{i=1}^{d}y_i\frac{\partial f(x)}{\partial x_i}1_{B(0,1)}(y)\right)\nu(x/\varepsilon,dy),\end{equation} the so-called (zero-drift) diffusion process with jumps  (see Section \ref{s2} for details). Here, as usual, $\varepsilon>0$
 is a small  parameter defined as a microstructure period  intended to tend to zero.
We assume that
  the coefficients $c(x)=(c_{ij}(x))_{1\leq i,j\leq d}$ and $\nu(x,dy)$ are periodic (say with period $\tau>0$ in all coordinate directions), $c(x)$ is symmetric and non-negative definite and $\nu(x,dy)$ is a  Borel kernel satisfying \begin{equation}\label{eq1.2}\sup_{x\in\R^{d}}\int_{\R^{d}}|y|^{2}\nu(x,dy)<\infty.\end{equation}
As we have commented, due to the rapidly oscillating nature of the coefficients of $\mathcal{A_\varepsilon}$,
 it is expected that the
microscopic behavior (heterogeneity)  can be averaged out to result a simpler model approximating
the behavior at the macroscopic level (the homogenization occurs). Indeed, as the main result of this paper, we prove that   as $\varepsilon\searrow0$, in the path space endowed with the Skorokhod topology, the homogenized process is a zero-drift Brownian motion  (or, equivalently, the homogenized operator is a zero-drift second-order
elliptic operator) determined by a covariance matrix (diffusion coefficient) of the form
$$\Bigg(\int_{[0,\tau]}c_{ij}(x)\pi(dx)+\int_{[0,\tau]}\int_{\R^{d}}y_iy_j\nu(x,dy)\pi(dx)\Bigg)_{1\leq
i,j\leq d},$$ where  $\pi(dx)$ is an invariant probability measure associated to the projection of $\process{F^{1}}$ on the $d$-dimensional torus $[0,\tau]^{d}$ (see Section \ref{s4} for details).
Note that, in view of the assumption in \eqref{eq1.2}, it is very natural to choose the diffusive scaling  in \eqref{eq1.1}. Namely,  $\process{F^{\varepsilon}}$ can be seen as $\{\varepsilon F_{\varepsilon^{-2}t}^{1}\}_{t\geq0}$ (both processes are generated by $\mathcal{A}_{\varepsilon}$) which naturally leads to a diffusive limit.

Further, observe that the process
$\process{F^{\varepsilon}}$ (operator $\mathcal{A}_{\varepsilon}$) can be seen as a classical (zero-drift) diffusion process (second-orderer elliptic operator) perturbed by a jump process (non-local  operator) which, in many situations, makes the whole  model more realistic. For example, such  processes (operators) are used to model population dynamics  also by taking
into account long-range effects (see \cite{car} and the references therein). In these models  the kernel $\nu(x,dy)$ is typically assumed to be a probability kernel  representing a  non-local
dispersal law of the population (for example, induced by  visual contact or chemical reaction).  Also, such models are nowadays a standard tool in representing the risk related to large market
movements (see \cite{cont}).
 Namely, due to the shortcomings of diffusion models in modeling these type of phenomena,  various  models (processes)
with jumps of type \eqref{eq1.1} were introduced, which allow for more realistic representation of
price dynamics and a greater flexibility in modeling.
 Observe that, from the practical point of view, in both contexts it is very natural to assume that the kernel $\nu(x,dy)$ has compact support or, equivalently, the corresponding diffusion with jumps  has uniformly bounded jumps.
 Hence, the condition in \eqref{eq1.2} is automatically satisfied.
 As we have already commented, the (small) parameter $\varepsilon>0$ characterizes the size of the cell of periodicity of the media. For example,  in the problem of forest migration  (see \cite{seed}) it represents the area of actively dispersed seeds (by animal cachers).
More precisely,   the process of active seed dispersion occurs on  scales much smaller than forest migration and, clearly,  it is natural to assume that the structure of the environment (cache areas) is more or less periodic. Also, note that the non-local effect in the model is reflected through disturbances
like wind, fire and avian dispersers.

The work in this paper is highly motivated
by  the works in \cite{benso-lions-book} and   \cite{bhat2}  in which the authors have considered the  homogenization of a diffusion   $\process{F^{\varepsilon}}$  generated by second-order
elliptic operator of the form
\begin{equation}\label{eq1.3}\mathcal{A}_\varepsilon f(x):=\frac{1}{\varepsilon}\sum_{i=1}^{d}b_i(x/\varepsilon)\frac{\partial f(x)}{\partial x_i}+\frac{1}{2}\sum_{i,j=1}^{d}c_{ij}(x/\varepsilon)\frac{\partial^{2}f(x)}{\partial x_i\partial x_j}.\end{equation}
Under the  assumptions  that the coefficients $b(x)=(b_i(x))_{1\leq i\leq d}$ and $c(x)=(c_{ij}(x))_{1\leq i,j\leq d}$ are periodic  and smooth enough (see Section \ref{s6} for details) and  $c(x)$ is symmetric and uniformly elliptic,     they have shown that
the  homogenized process (operator) of the diffusion   $\{ F^{\varepsilon}_{t}-\varepsilon^{-1}\bar{b}t\}_{t\geq0}$
(generator  $\mathcal{A}_\varepsilon f(x) -\frac{1}{\varepsilon}\sum_{i=1}^{d}\bar{b}_i\frac{\partial f(x)}{\partial x_i}$)
is a zero-drift Brownian motion (zero-drift second-order
elliptic operator) determined by a covariance matrix (diffusion coefficient)
 of the form
 \begin{equation}\label{eq1.4}\bar{\Sigma}:=\left(\int_{[0,\tau]^{d}}\sum_{k,l=1}^{d}\left(\delta_{ki}-\frac{\partial\beta_i(x)}{\partial x_k}\right)c_{kl}(x)\left(\delta_{lj}-\frac{\partial\beta_j(x)}{\partial x_l}\right)\pi(dx)\right)_{1\leq i,j\leq d}.\end{equation}
 Here, $\bar{b}:=(\bar{b}_i)_{1\leq i\leq d}$ for $\bar{b}_i:=\int_{[0,\tau]}b_i(x)\pi(dx)$,  $\pi(dx)$ is the unique invariant probability measure corresponding to the projection of $\process{F^{1}}$ on the $d$-dimensional torus $[0,\tau]^{d}$ (see Section \ref{s4} for details),
and  $\beta_i\in C^{2}(\R^{d})$, $i=1,\ldots,d,$ are unique $\tau$-periodic solutions to $\mathcal{A}_1\beta_i(x)=b_i(x)-\bar{b}_i$
(see Section \ref{s6} for details). For more
on the homogenization of local operators (methods and applications)   we refer the readers to the monographs \cite{Allaire},
\cite{benso-lions-book}, \cite{kozlov}, \cite{Piatnitski} and \cite{tartarII}.

In the closely related paper \cite{rhodes} the authors have considered the homogenization  of a one-dimensional  diffusion with jumps in  ergodic random environment $(\Omega,\mathcal{F},\mathbb{P},\{\tau_x\}_{x\in\R})$, generated by an operator of the form (in a fixed random environment $\omega\in\Omega$) \begin{align}\label{eq1.5}\mathcal{A}_\varepsilon^{\omega} f(x):=&\frac{1}{\varepsilon}b(\tau_{x/\varepsilon}\omega)f'(x)+
\frac{1}{2}c(\tau_{x/\varepsilon}\omega)f''(x)\nonumber\\&+\frac{1}{\varepsilon^{2}}\int_{\R}\left(f(\varepsilon y+x)-f(x)- \varepsilon y f'(x)1_{B(0,1)}(y)\right)\nu(\tau_{x/\varepsilon}\omega,dy),\end{align}
where    $\{\tau_x\}_{x\in\R}$ is  an ergodic group of measure
preserving transformations acting on the probability space $(\Omega,\mathcal{F},\mathbb{P})$. By assuming that $\int_{\R}|y|^{2}\nu(\omega,dy)<\infty$ for all $\omega\in\Omega$ (and certain regularity conditions on the coefficients of $\mathcal{A}_\varepsilon^{\omega}$), they have shown that the homogenized process  is a Brownian motion. Observe that if we remove the environment in \eqref{eq1.5}, the underlying process  becomes a L\'evy process (diffusion with jumps with constant coefficients) whose homogenization we discuss in Corollary \ref{c5.2}. Let us remark that the authors have also discussed the case when  $\int_{\R}|y|^{2}\nu(\omega,dy)=\infty$ for all $\omega\in\Omega$. Under a superdiffusive scaling (and certain regularity conditions on the coefficients of $\mathcal{A}_\varepsilon^{\omega}$), they have obtained that the homogenized  process is a pure-jump L\'evy process (it is generated by a purely non-local operator). Related to this, the homogenization of a diffusion with jumps  generated by
\begin{equation}\label{eq:x}\mathcal{A}_\varepsilon f(x):=\frac{1}{\varepsilon^{\alpha_0}}\int_{\R^{d}}\left(f(\varepsilon y+x)-f(x)-\varepsilon\sum_{i=1}^{d}y_i\frac{\partial f(x)}{\partial x_i}1_{B(0,1)}(y)\right)\frac{\beta(x/\varepsilon)}{|y|^{d+\alpha(x/\varepsilon)}}dy,\end{equation} where
$\alpha:\R^{d}\longrightarrow(0,2)$ and $\beta:\R^{d}\longrightarrow(0,\infty)$ are periodic, continuously differentiable and $\{x\in\R^{d}:\alpha(x)=\alpha_0:=\inf_{x\in\R^{d}}\alpha(x)\}$ has positive Lebesgue measure, has been considered in \cite{franke-periodic, franke-periodicerata} (by probabilistic methods) and, in the case of constant parameter $\alpha(x)$, in
\cite{ari} and \cite{schwabI} (by an analytic method, two-scale convergence method, which strongly relies on the fact that the parameter $\alpha(x)$ is constant,  scaling property of the underlying jump kernel and regularity properties of the solution of the corresponding boundary value problem). Note that in this situation the jump kernel $\nu(x,dy):=\beta(x)|y|^{-d-\alpha(x)}dy$ does not  satisfy anymore the  relation in \eqref{eq1.2}. Accordingly, the superdiffusive scaling $\varepsilon^{-\alpha_0}$ is used and the corresponding homogenized process  cannot   be anymore a Brownian motion  (the jumps are too big) but  it is a pure-jump L\'evy process (stable L\'evy-process) generated by
$$\mathcal{A}_0 f(x):=\int_{\R^{d}}\left(f( y+x)-f(x)-\sum_{i=1}^{d}y_i\frac{\partial f(x)}{\partial x_i}1_{B(0,1)}(y)\right)\frac{\beta_0}{|y|^{d+\alpha_0}}dy$$ for some constant  $\beta_0>0$.
 Finally, let us also remark that the homogenization of the operator in \eqref{eq:x} (with constant parameter $\alpha(x)$) has been also discussed in: (i) the  quasi-periodic and almost periodic setting (see \cite{ari2}) and (ii) the case when  $\{\beta(x)\}_{x\in\R^{d}}$ forms a stationary ergodic family defined on some probability space $(\Omega,\mathcal{F},\mathbb{P})$ endowed with  ergodic group of measure
preserving transformations $\{\tau_x\}_{x\in\R^{d}}$ acting on  $(\Omega,\mathcal{F},\mathbb{P})$ (see \cite{sch2}).
 Again, the techniques and methods used to resolve these problems
 strongly rely on the fact that the parameter $\alpha(x)$ is constant,  scaling property of the underlying jump kernel and regularity properties of the solution of the corresponding boundary value problem.

The main techniques used  in  \cite{benso-lions-book} and   \cite{bhat2} (which strongly rely on the fact that the underlying process is a diffusion  with smooth coefficients) are based on proving the  convergence as $\varepsilon\searrow0$ of finite-dimensional distributions of the   diffusion  $\process{F^{\varepsilon}}$ and functional central limit theorems for stationary ergodic sequences constructed from this process. On the other hand, our approach
 is through the  characteristics of  semimartingales. More precisely, to determine the homogenized process, we reduce the problem to the convergence of   semimartingale characteristics of the   process $\{ F^{\varepsilon}_{t}\}_{t\geq0}$.
 Namely,  by using the facts that $\{ F^{\varepsilon}_{t}\}_{t\geq0}$ is a  semimartingale whose characteristics are given in terms of the coefficients of the corresponding generator $\mathcal{A}_\varepsilon$
   (see \cite[Lemma 3.1 and Theorem 3.5]{rene-holder})  and the regularity assumptions imposed upon the coefficients  $c(x)$ and $\nu(x,dy)$, we show that the characteristics of $\{ F^{\varepsilon}_{t}\}_{t\geq0}$ as $\varepsilon\searrow0$  converge (in probability) to the   characteristics of a certain Brownian motion, which,
   according to \cite[Theorem VIII.2.17]{jacod}, proves the desired result.

The remainder of the paper is organized as follows.
In Section \ref{s2}, we give some preliminaries on diffusions with jumps and in Section \ref{s3} we state the main result of the paper, Theorem \ref{tm3.1}. In Section \ref{s4}, we provide some auxiliary results concerning periodic diffusions with jumps,
 in Section \ref{s5} we prove Theorem \ref{tm3.1} and in Section \ref{s6} we discuss the non-zero drift case. Finally, in Section \ref{s7}, we give an application of the presented homogenization results to the long-time behavior (transience, recurrence and ergodicity) of periodic diffusions with small jumps.

\section{Preliminaries on Diffusions with Jumps} \label{s2}
 Let
$(\Omega,\mathcal{F},\{\mathbb{P}_{x}\}_{x\in\R^{d}},\process{\mathcal{F}},\process{\theta},\process{M})$, denoted by $\process{M}$
in the sequel, be a Markov process with  state space
$(\R^{d},\mathcal{B}(\R^{d}))$, where $d\geq1$ and
$\mathcal{B}(\R^{d})$ denotes the Borel $\sigma$-algebra on
$\R^{d}$. Due to the Markov property, the associated family of linear operators $\process{P}$ on
$B_b(\R^{d})$ (the space of bounded and Borel measurable functions),
defined by $$P_tf(x):= \mathbb{E}_{x}[f(M_t)],\quad t\geq0,\
x\in\R^{d},\ f\in B_b(\R^{d}),$$  forms a \emph{semigroup}  on the
Banach space $(B_b(\R^{d}),||\cdot||_\infty)$, that is, $P_s\circ
P_t=P_{s+t}$ and $P_0=I$ for all $s,t\geq0$. Here, $\mathbb{E}_x$ stands for the expectation with respect to $\mathbb{P}_x(d\omega)$, $x\in\R^{d}$, and
$||\cdot||_\infty$ denotes the supremum norm on the space
$B_b(\R^{d})$. Moreover, the semigroup $\process{P}$ is
\emph{contractive} ($||P_tf||_{\infty}\leq||f||_{\infty}$
for all $t\geq0$ and all $f\in B_b(\R^{d})$) and \emph{positivity
preserving} ($P_tf\geq 0$ for all $t\geq0$ and all $f\in
B_b(\R^{d})$ satisfying $f\geq0$). The \emph{infinitesimal generator}
$(\mathcal{A}^{b},\mathcal{D}_{\mathcal{A}^{b}})$ of the semigroup
$\process{P}$ (or of the process $\process{M}$) is a linear operator
$\mathcal{A}^{b}:\mathcal{D}_{\mathcal{A}^{b}}\longrightarrow B_b(\R^{d})$
defined by
$$\mathcal{A}^{b}f:=
  \lim_{t\longrightarrow0}\frac{P_tf-f}{t},\quad f\in\mathcal{D}_{\mathcal{A}^{b}}:=\left\{f\in B_b(\R^{d}):
\lim_{t\longrightarrow0}\frac{P_t f-f}{t} \ \textrm{exists in}\
||\cdot||_\infty\right\}.
$$ We call $(\mathcal{A}^{b},\mathcal{D}_{\mathcal{A}^{b}})$ the \emph{$B_b$-generator} for short.
A Markov process $\process{M}$ is said to be a \emph{Feller process}
if its corresponding  semigroup $\process{P}$ forms a \emph{Feller
semigroup}. This means that
\begin{itemize}
  \item [(i)] $\process{P}$ enjoys the \emph{Feller property}, that is,  $P_t(C_\infty(\R^{d}))\subseteq C_\infty(\R^{d})$ for all $t\geq0$;
  \item [(ii)] $\process{P}$ is \emph{strongly continuous}, that is, $\lim_{t\longrightarrow0}||P_tf-f||_{\infty}=0$ for all $f\in
  C_\infty(\R^{d})$.
\end{itemize}
 Here, $C_\infty(\R^{d})$ denotes
the space of continuous functions vanishing at infinity.  Note that
every Feller semigroup $\process{P}$  can be uniquely extended to
$B_b(\R^{d})$ (see \cite[Section 3]{rene-conserv}). For notational
simplicity, we denote this extension again by $\process{P}$. Also,
let us remark that every Feller process possesses the strong Markov
property and has c\`adl\`ag sample paths (see  \cite[Theorems 3.4.19 and
3.5.14]{jacobIII}).   Further,
in the case of Feller processes, we call
$(\mathcal{A}^{\infty},\mathcal{D}_{\mathcal{A}^{\infty}}):=(\mathcal{A}^{b},\mathcal{D}_{\mathcal{A}^{b}}\cap
C_\infty(\R^{d}))$ the \emph{Feller generator} for short. Observe
that, in this case, $\mathcal{D}_{\mathcal{A}^{\infty}}\subseteq
C_\infty(\R^{\bar{d}})$ and
$\mathcal{A}^{\infty}(\mathcal{D}_{\mathcal{A}^{\infty}})\subseteq
C_\infty(\R^{d})$. If the set of smooth functions
with compact support $C_c^{\infty}(\R^{d})$ is contained in
$\mathcal{D}_{\mathcal{A}^{\infty}}$, that is, if the Feller
generator
$(\mathcal{A}^{\infty},\mathcal{D}_{\mathcal{A}^{\infty}})$ of the
Feller process $\process{M}$ satisfies
    \begin{description}
      \item[(\textbf{C1})]
      $C_c^{\infty}(\R^{d})\subseteq\mathcal{D}_{\mathcal{A}^{\infty}}$,
    \end{description}
 then, according to \cite[Theorem 3.4]{courrege-symbol},
$\mathcal{A}^{\infty}|_{C_c^{\infty}(\R^{d})}$ is a \emph{pseudo-differential
operator}, that is, it can be written in the form
\begin{equation}\label{eq2.1}\mathcal{A}^{\infty}|_{C_c^{\infty}(\R^{d})}f(x) = -\int_{\R^{d}}q(x,\xi)e^{i\langle \xi,x\rangle}
\hat{f}(\xi) d\xi,\end{equation}  where $\hat{f}(\xi):=
(2\pi)^{-d} \int_{\R^{d}} e^{-i\langle\xi,x\rangle} f(x) dx$ denotes
the Fourier transform of the function $f(x)$. The function $q :
\R^{d}\times \R^{d}\longrightarrow \CC$ is called  the \emph{symbol}
of the pseudo-differential operator. It is measurable and locally
bounded in $(x,\xi)$ and continuous and negative definite as a
function of $\xi$. Hence, by \cite[Theorem 3.7.7]{jacobI}, the
function $\xi\longmapsto q(x,\xi)$ has for each $x\in\R^{d}$ the
following L\'{e}vy-Khintchine representation \begin{equation}\label{eq2.2}q(x,\xi) =a(x)-
i\langle \xi,b(x)\rangle + \frac{1}{2}\langle\xi,c(x)\xi\rangle +
\int_{\R^{d}}\left(1-e^{i\langle\xi,y\rangle}+i\langle\xi,y\rangle1_{B(0,1)}(y)\right)\nu(x,dy),\end{equation}
where  $a(x)$ is a non-negative Borel measurable function, $b(x)$ is
an $\R^{d}$-valued Borel measurable function,
$c(x):=(c_{ij}(x))_{1\leq i,j\leq d}$ is a symmetric non-negative
definite $d\times d$ matrix-valued Borel measurable function
 and $\nu(x,dy)$ is a Borel kernel on $\R^{d}\times
\mathcal{B}(\R^{d})$, called the \emph{L\'evy measure}, satisfying
$$\nu(x,\{0\})=0\quad \textrm{and} \quad \int_{\R^{d}}\left(1\wedge
|y|^{2}\right)\nu(x,dy)<\infty,\quad x\in\R^{d}.$$
The quadruple
$(a(x),b(x),c(x),\nu(x,dy))$ is called the \emph{L\'{e}vy quadruple}
of the pseudo-differential operator
$\mathcal{A}^{\infty}|_{C_c^{\infty}(\R^{d})}$ (or of the symbol $q(x,\xi)$).
Let us remark that the local boundedness of $q(x,\xi)$  implies the local boundedness of the corresponding coefficients
 (see \cite[Lemma 2.1 and Remark 2.2]{rene-holder}).
In the sequel, we assume the following condition on the symbol
$q(x,\xi)$:
\begin{description}
  \item[(\textbf{C2})] $q(x,0)=a(x)=0$ for all $x\in\R^{d}.$
\end{description}
This condition is closely related to the conservation property of $\process{M}$.
Namely, under  the assumption that the coefficients of $q(x,\xi)$ are uniformly bounded (which is certainly the case in the periodic setting), (\textbf{C2}) implies that $\process{M}$ is \emph{conservative}, that is, $\mathbb{P}_{x}(M_t\in\R^{d})=1$ for all $t\geq0$ and all
$x\in\R^{d}$ (see \cite[Theorem 5.2]{rene-conserv}). Further, note that by combining (\ref{eq2.1}), (\ref{eq2.2}) and (\textbf{C2}),  $\mathcal{A}^{\infty}|_{C_c^{\infty}(\R^{d})}$ has a representation as an integro-differential operator (\ref{eq1.1}).
Also note that in the case when the symbol $q(x,\xi)$ does not depend
on the variable $x\in\R^{d}$, $\process{M}$ becomes a \emph{L\'evy
process}, that is, a stochastic process   with stationary and
independent increments and c\`adl\`ag sample paths. Moreover, unlike Feller
processes, every L\'evy process is uniquely and completely
characterized through its corresponding symbol (see \cite[Theorems
7.10 and 8.1]{sato-book}). According to this, it is not hard to
check that every L\'evy process satisfies conditions
(\textbf{C1}) and (\textbf{C2})  (see \cite[Theorem 31.5]{sato-book}).
Thus, the class of processes we consider in this paper contains a subclass
 of L\'evy processes. Throughout this paper, the symbol $\process{F}$ denotes a Feller
process satisfying conditions (\textbf{C1}) and (\textbf{C2}). Such a
process is called a \emph{diffusion process with jumps}.
If $\nu(x,dy)=0$ for all $x\in\R^{d}$, then $\process{F}$ is just called a \emph{diffusion process}.  Note that this definition agrees with the standard definition of diffusions (see \cite{rogersI}).
 For more
on diffusions with jumps  we refer the readers to the monograph
\cite{rene-bjorn-jian}.

\section{Main Results}
\label{s3} Before stating the main results of this paper, we  introduce some notation we need.
Let $\tau:=(\tau_1,\ldots, \tau_{d})\in (0,\infty)^{d}$ be
fixed and let $\tau\ZZ^{d}:=\tau_1\ZZ\times\ldots\times\tau_{d} \ZZ.$
For $x\in\R^{d}$, define
$$x_\tau:=\{y\in\R^{d}:x-y\in\tau\ZZ^{d}\}\quad\textrm{and}\quad
\R^{d}/\tau\ZZ^{d}:=\{x_\tau:x\in\R^{d}\}.$$ Clearly,
$\R^{d}/\tau\ZZ^{d}$ is obtained
 by identifying the opposite
faces of $[0,\tau]:=[0,\tau_1]\times\ldots\times[0,\tau_{d}]$.
Next,
let
 $\Pi_{\tau} : \R^{d}\longrightarrow [0,\tau]$, $\Pi_{\tau}(x):=x_\tau$, be the covering map. A
function $f:\R^{d}\longrightarrow\R$ is called $\tau$-periodic if
$$f\circ\Pi_\tau(x)=f(x),\quad x\in\R^{d}.$$
For notational convenience, from now on we will omit the subscript $\tau$ and simply write $x$ instead of $x_\tau$.
Now, we are in position to state the main results of this paper, the proofs of which  are given in
Section \ref{s5}.
\begin{theorem}\label{tm3.1}Let  $\process{F}$ be a $d$-dimensional diffusion with jumps with  semigroup $\process{P}$, symbol $q(x,\xi)$ and L\'evy triplet $(0,c(x),\nu(x,dy))$,
  which satisfy:
 \begin{description}
     \item [(\textbf{C3})] $\process{F}$ possesses a transition
density function $p(t,x,y)$, that is,
$$P_tf(x)=\int_{\R^{d}}f(y)p(t,x,y)dy,\quad t>0,\ x\in\R^{d},\ f\in B_b(\R^{d}),$$ such that  $(x,y)\longmapsto p(t,x,y)$ is continuous and $p(t,x,y)>0$
for all $t>0$ and all $x,y\in \R^{d}$;
     \item [(\textbf{C4})] the function $x\longmapsto q(x,\xi)$ (or, equivalently, the  coefficients $c(x)$ and $\nu(x,dy)$) is $\tau$-periodic for all $\xi\in\R^{d}$;
     \item [(\textbf{C5})]  $\displaystyle\sup_{x\in[0,\tau]}\int_{\R^{d}}|y|^{2}\nu(x,dy)<\infty;$
   \item [(\textbf{C6})] $\displaystyle\lim_{\varepsilon\longrightarrow0}\sup_{x\in[0,\tau]}\left|\int_{B(0,\varepsilon^{-1})\setminus B(0,1)}y_i\,\nu(x,dy)\right|=0\quad \textrm{for all}\quad i=1,\ldots,d$.
   \end{description}
Then, for any initial
distribution  of $\process{F}$,
\begin{equation}\label{eq3.1}\left\{\varepsilon F_{\varepsilon^{-2}t}\right\}_{t\geq0}\xrightarrow[\varepsilon\searrow0]{d}\process{W}.\end{equation}
 Here, $\stackrel{\hbox{\scriptsize{$\textrm{d}$}}}{\longrightarrow}$ denotes the convergence in the space of c\`adl\`ag functions
endowed with the Skorohod  topology,
$\process{W}$ is a $d$-dimensional zero-drift Brownian motion starting from the origin and
determined by a covariance matrix  of the form
\begin{equation} \label{eq3.2} \Sigma:=\Bigg(\int_{[0,\tau]}c_{ij}(z)\pi(dz)+\int_{[0,\tau]}\int_{\R^{d}}y_iy_j\nu(z,dy)\pi(dz)\Bigg)_{1\leq
i,j\leq d}
\end{equation} and  $\pi(dx)$ is an invariant probability measure associated to the projection of $\process{F}$, with respect to $\Pi_{\tau}(x)$, on $[0,\tau]$ (see Section \ref{s4} for details).
\end{theorem}

\begin{remark}\label{r3.2}{\rm
In Theorem \ref{tm3.1} we assume that the drift term vanishes, which is a
crucial assumption.  Namely,
as we commented in the first section, in  \cite{benso-lions-book} and  \cite{bhat2} the authors have considered the homogenization of a $d$-dimensional diffusion  $\{\varepsilon F_{\varepsilon^{-2}t}\}_{t\geq0}$  generated by a second-order elliptic operator \eqref{eq1.3} with not necessarily vanishing drift term.  More precisely, by using different approaches (which strongly rely on the fact that the underlying process is a diffusion with smooth coefficients), they have proved
that: (i) the family of  diffusions $\left\{\varepsilon F_{\varepsilon^{-2}t}-\varepsilon^{-1}\bar{b}t\right\}_{t\geq0}$, $\varepsilon>0$, is tight  and (ii) the finite-dimensional distributions of  $\left\{\varepsilon F_{\varepsilon^{-2}t}-\varepsilon^{-1}\bar{b}t\right\}_{t\geq0}$ converge to those  of $\process{\bar{W}}$, as $\varepsilon\searrow0$. Recall, $\bar{b}=(\int_{[0,\tau]}b_i(x)\pi(dx))_{1\leq i\leq d}$ and  $\process{\bar{W}}$ is a zero-drift Brownian motion determined by the covariance matrix given in \eqref{eq1.4}. Note that, due to  Prokhorov's theorem, (i) and (ii) are equivalent to \begin{equation}\label{eq3}\left\{\varepsilon F_{\varepsilon^{-2}t}-\varepsilon^{-1}\bar{b}t\right\}_{t\geq0}\xrightarrow[\varepsilon\searrow0]{d}\process{\bar{W}}.\end{equation}
On the other hand, the approach through the characteristics of semimartingales gives only sufficient conditions for the convergence in the Skorokhod topology and it is too weak to deal with the non-zero drift case. Namely,  by using this approach, in Section \ref{s6} we prove  that the relation in \eqref{eq3} (for a diffusion with jumps  satisfying (\textbf{C3})-(\textbf{C6})) holds if, and only if, $b(x)=\bar{b}$  ($dx$-a.e.).
}
\end{remark}

 \begin{remark}\label{r3.3}
{\rm Note that  (non-degenerate) diffusions  always satisfy the assumptions in   (\textbf{C3}) (see  \cite{rogersI} and
\cite[Theorem A]{sheu}). Hence, Theorem \ref{tm3.1} generalizes the results related to diffusions (with zero-drift term), presented in
\cite{benso-lions-book} and \cite{bhat2}.}
\end{remark}

\begin{remark}\label{r3.4} {\rm In (\textbf{C3})  we assume the existence, continuity (in space variables) and strict positivity of the transition density function $p(t,x,y)$ of $\process{F}$.  According to \cite[Theorem 2.6]{sandric-tams}, under  the assumptions that $q(x,\xi)$ has bounded coefficients  (which is certainly the case in the periodic setting) and $q(x,\xi)={\rm Re}\,q(x,\xi)$ for all $x,\xi\in\R^{d}$ (or, equivalently,  $b(x)=0$
and $\nu(x,dy)$ is a symmetric measure for all $x\in\R^{d}$),
the existence  of $p(t,x,y)$  follows from
\begin{equation}\label{eq3.3}\int_{\R^{d}}\exp\left[-t\, \inf_{x\in\R^{d}}q(x,\xi)\right]d\xi<\infty,\quad t>0.\end{equation}
  According to \cite{friedman} and \cite[Theorems 7.10 and
8.1]{sato-book}, in the L\'evy process and diffusion case, in order to ensure the
existence of a transition density function,  the  assumption that the symbol is real is not essential.
 Further, note that $(\ref{eq3.3})$ (together with the assumptions that the symbol  has bounded coefficients and it is real)  also implies the continuity of $(x,y)\longmapsto p(t,x,y)$ for all $t>0$. First, under these assumptions,
\cite[Theorems 1.1 and 1.2]{rene-wang-feller} entail that
$$\sup_{x\in\R^{d}}\left|\mathbb{E}^{x}\left[e^{i\xi(F_t-x)}\right]\right|\leq\exp\left[-\frac{t}{16}\inf_{x\in\R^{d}}\it{q}(x,\rm{2}\xi)\right],\quad
t>0,\ \xi\in\R^{d},$$ and that the  function
$x\longmapsto\mathbb{E}^{x}\left[e^{i\xi(F_t-x)}\right]$ is
continuous for all $\xi\in\R^{d}$.
Consequently,  \cite[Proposition
2.5]{sato-book} states that $p(t,x,y)$ exists and it is given by
$$p(t,x,y)=(2\pi)^{-d}\int_{\R^{d}}e^{-i\xi(y-x)}\mathbb{E}^{x}\left[e^{i\xi(F_t-x)}\right]d\xi,\quad
t>0,\ x,y\in\R^{d}.$$ Finally, the continuity of
$(x,y)\longmapsto p(t,x,y)$, $t>0$, follows directly by employing the dominated convergence
theorem. Observe that \eqref{eq3.3} is trivially satisfied in the case when the diffusion coefficient is uniformly elliptic.
On the other hand, the strict positivity of the transition density function
$p(t,x,y)$ is a more complex problem. In the  L\'evy process and
diffusion case this problem has been considered in \cite{brockett},
\cite{rajput}, \cite{friedman}, \cite{sharpe}  and \cite{sheu}. In
the general case, when  $\process{F}$ is a diffusion with jumps with  bounded coefficients and real symbol satisfying \ $(\ref{eq3.3})$, the best we were able to prove is given in \cite[Proposition 6.1]{turbul} and it reads as follows:
 $\process{F}$ possesses a transition
density function $p(t,x,y)$ such that for any
$t_0>0$ and any $n\geq1$ there exists $\varepsilon(t_0)>0$ such that $p(t,x,y+x)>0$
for all $t\in[nt_0,n(t_0+1)]$, all $x\in\R^{d}$ and all
$y\in B(0,n\varepsilon(t_0))$.}
 \end{remark}

As a direct  consequence of Theorem \ref{tm3.1} and \cite[Theorem 7.1]{rene-bjorn-jian} we get the following result.
\begin{corollary}\label{c3.5}
Let $\process{F}$ be a $d$-dimensional diffusion with jumps with semigroup $\process{P}$ and symbol $q(x,\xi)$, satisfying conditions
(\textbf{C3})-(\textbf{C6}). Then, for $\varepsilon>0$, $\{\varepsilon F_{\varepsilon^{-2}t}\}_{t\geq0}$ is a diffusion with jumps determined by symbol of the form
 $q^{\varepsilon}(x,\xi):=\varepsilon^{-2}q(x/\varepsilon,\varepsilon\xi).$
 Further, let $\process{P^{\varepsilon}}$ and $(\mathcal{A}^{\infty}_{\varepsilon},\mathcal{D}_{\mathcal{A}^{\infty}_{\varepsilon}})$, $\varepsilon>0$, and $\process{P^{0}}$ and $(\mathcal{A}^{\infty}_{0},\mathcal{D}_{\mathcal{A}^{\infty}_{0}})$ be the corresponding Feller semigroups and generators of $\{\varepsilon F_{\varepsilon^{-2}t}\}_{t\geq0}$, $\varepsilon>0$, and  $\process{W}$, respectively.
Then,
$$\lim_{\varepsilon\longrightarrow0}\sup_{t\leq t_0}||P_t^{\varepsilon}f-P_t^{0}f||_{\infty}=0,\quad t_0\geq0,\ f\in C_{\infty}(\R^{d}),$$ and
for each $f\in C_c^{\infty}(\R^{d})$ there exist functions $f_\varepsilon\in\mathcal{D}_{\mathcal{A}^{\varepsilon}}$, $\varepsilon>0$, such that $$\lim_{\varepsilon\longrightarrow0}\left(||f_\varepsilon-f||_{\infty}+||\mathcal{A}^{\infty}_{\varepsilon}f_\varepsilon-\mathcal{A}^{\infty}_{0}f||_{\infty}\right)=0.$$
\end{corollary}

\begin{definition}\label{d3.6}Let $O\subseteq\R^{d}$, $d\geq2$, be a connected open set. The set $O$ satisfies the \emph{Poincar\'{e} cone
condition} at $x\in\partial O$ if there exists a cone $C$ based at $x$ with opening angle $\varphi>0$, and
$r > 0$ such that $C\cap B(x,r)\subseteq O^{c}$. Here, $\partial O$ denotes the boundary of the set $O$.
\end{definition}

 \begin{theorem}\label{tm3.7}Let  $\process{F}$ be a $d$-dimensional diffusion with jumps  satisfying conditions (\textbf{C3})-(\textbf{C6}) and let
 $\process{W}$ be a   zero-drift Brownian motion
determined by the covariance matrix $\Sigma$ (given in \eqref{eq3.2}). Further, let $O\subseteq\R^{d}$, $d\geq2$, be an open set such that $\partial O$ has Lebesgue measure zero and $O$ and ${\rm Int}\,O^{c}$ (interior of the set $O^{c}$) are connected and satisfy the Poincar\'{e} cone
condition on $\partial O$. If $d=1$, just assume that $O\subseteq\R$ is open.
For $\varepsilon>0$, define $T^{\varepsilon}:=\inf\{t>0:\varepsilon F_{\varepsilon^{-2}t}\notin O\}$ and  $T^{0}:=\inf\{t>0:W_{t}\notin O\}$.
 Then, for any initial
distribution  of $\process{F}$,
\begin{equation}\label{eq3.7}\varepsilon F_{\varepsilon^{-2} T^{\varepsilon}}\xrightarrow[\varepsilon\searrow0]{d}W_{T^{0}}\end{equation}
and
\begin{equation}\label{eq3.8}\left\{\varepsilon F_{\varepsilon^{-2}(t\wedge T^{\varepsilon})}\right\}_{t\geq0}\xrightarrow[\varepsilon\searrow0]{d}\{W_{t\wedge T^{0}}\}_{t\geq0}.\end{equation}
\end{theorem}
As a consequence of Theorem \ref{tm3.7} and \cite[Theorem 6.5.1]{fried} we get the following well-known result.
\begin{corollary}\label{c3.8}
Let  $\process{F}$ be a $d$-dimensional  diffusion with vanishing drift coefficient and uniformly elliptic diffusion coefficient $c(x)$,  satisfying conditions (\textbf{C3}) and  (\textbf{C4}).
Further, let
 $\process{W}$ be a   zero-drift Brownian motion
determined by the covariance matrix $\Sigma$ (given in \eqref{eq3.2}), let $O\subseteq\R^{d}$ be an open bounded set with $C^{2}$ boundary  satisfying the assumptions from Theorem \ref{tm3.7} and let $a:O\cup\partial O\longrightarrow(-\infty,0]$ be an arbitrary H\"{o}lder continuous function.
 For  $\varepsilon>0$, let $(\mathcal{A}^{\infty}_\varepsilon,\mathcal{D}_{\mathcal{A}^{\infty}_\varepsilon})$ and  $(\mathcal{A}^{\infty}_0,\mathcal{D}_{\mathcal{A}^{\infty}_0})$   be the Feller generators   of $\{\varepsilon F_{\varepsilon^{-2}t}\}_{t\geq0}$ and $\process{W}$, respectively, that is,
 $$\mathcal{A}^{\infty}_\varepsilon f(x)=\frac{1}{2}\sum_{i,j=1}^{d}c_{ij}(x/\varepsilon)\frac{\partial^{2}f}{\partial x_i\partial x_j}(x)\quad \textrm{and}\quad \mathcal{A}^{\infty}_0 f(x)=\frac{1}{2}\sum_{i,j=1}^{d}\int_{[0,\tau]}c_{ij}(z)\pi(dz)\frac{\partial^{2}f}{\partial x_i\partial x_j}(x).$$
 Then, for any $\varepsilon\geq0$, any  H\"{o}lder continuous function  $f:O\cup\partial O\longrightarrow\R$ and any continuous function $g:\partial O\longrightarrow\R$,
 there exists a unique solution $u_\varepsilon(x)$ to the following Dirichlet problem \begin{align*}\mathcal{A}^{\infty}_\varepsilon u_\varepsilon(x)+a(x)u_\varepsilon(x)&=f(x),\quad x\in O\\
 u_\varepsilon(x)&=g(x),\quad x\in\partial O.\end{align*} Moreover, $u_\varepsilon(x)$ converges (pointwise) to $u(x)$ as $\varepsilon\searrow0.$
\end{corollary}

\section{Auxiliary Results}\label{s4}

We start this section with the following observation. Let  $\process{F}$ be a $d$-dimensional diffusion with jumps with  transition density function $p(t,x,y)$ and L\'evy triplet $(b(x),c(x),\nu(x,dy))$. Then, the coefficients $b(x)$, $c(x)$ and $\nu(x,dy)$ are $\tau$-periodic if, and only if,
 the function $x\longmapsto p(t,x,x+y)$ is $\tau$-periodic for all $t>0$ and all $y\in\R^{d}$.    The sufficiency
follows directly from  \cite[the proof of Theorem 4.5.21]{jacobI}. To prove the
necessity, first recall that there exists a suitable enlargement of the stochastic basis    $(\Omega,\mathcal{F},\{\mathbb{P}_{x}\}_{x\in\R^{d}},\process{\mathcal{F}},\process{\theta})$, say $(\widetilde{\Omega},\mathcal{\widetilde{F}},\{\mathbb{\widetilde{P}}_{x}\}_{x\in\R^{d}},\process{\mathcal{\widetilde{F}}},\process{\widetilde{\theta}})$, on which $\process{F}$ is
the solution to the following stochastic differential equation
\begin{align}\label{eq4.1}F_t=&x+\int_0^{t}b(F_{s-})ds+\int_0^{t}c(F_{s-})d\widetilde{W}_s\nonumber\\&+\int_0^{t}\int_{\R\setminus\{0\}}k(F_{s-},z)1_{\{u:|k(F_{s-},u)|\leq1\}}(z)\left(\widetilde{\mu}(\cdot,ds,dz)-ds\widetilde{N}(dz)\right)\nonumber\\&+\int_0^{t}\int_{\R\setminus\{0\}}k(F_{s-},z)1_{\{u:|k(F_{s-},u)|>1\}}(z)\,\widetilde{\mu}(\cdot,ds,dz),
\end{align}
where $\process{\widetilde{W}}$ is a $d$-dimensional Brownian motion, $\widetilde{\mu}(\omega,ds,dz)$
is a Poisson random measure with compensator (dual
predictable projection) $ds\widetilde{N}(dz)$ and $k:\R^{d}\times\R\setminus\{0\}\longrightarrow\R^{d}$ is a Borel measurable function  satisfying
$$\widetilde{\mu}(\omega,ds,k(F_{s-}(\omega),\cdot)\in dy)=\sum_{s: \Delta F_s(\omega)\neq0}\delta_{(s,\Delta F_s(\omega))}(ds,dy)$$
and
$$ds\widetilde{N}(k(F_{s-}(\omega),\cdot)\in dy)=ds\,\nu(F_{s-}(\omega),dy)$$ (see \cite[Theorem 3.5]{rene-holder} and \cite[Theorem 3.33]{cinlar}).
Further,  $\process{F}$ has the same transition function on the starting and enlarged stochastic basis. Thus, because of the $\tau$-periodicity of $b(x)$, $c(x),$ and $\nu(x,dy))$, directly from (\ref{eq4.1}) we read that $\mathbb{P}_{x+\tau}(F_t\in dy)=\mathbb{P}_{x}(F_t+\tau\in dy)$ for all $t\geq0$ and all $x\in\R^{d}$, which proves the assertion.

Now, according to the above observation, we easily deduce that if $\process{F}$ is a $d$-dimensional diffusion with jumps with semigroup $\process{P}$, transition density function $p(t,x,y)$ and $\tau$-periodic L\'evy triplet $(b(x),c(x),\nu(x,dy))$, then $\process{P}$ preserves the class of all bounded Borel measurable $\tau$-periodic functions, that is, the function $x\longmapsto P_tf(x)$ is $\tau$-periodic for all $t\geq0$ and all $\tau$-periodic $f\in B_b(\R^{d})$. Consequently, \cite[Proposition 3.8.3]{vasili-book} entails that  $F_t^{\tau}:=\Pi_{\tau}(F_t),$
$t\geq0$, is a
 Markov process on $([0,\tau],\mathcal{B}([0,\tau])$ with positivity preserving contraction semigroup $\process{P^{\tau}}$
(on the space $(B_b([0,\tau]),||\cdot||_\infty)$) given by
$$P_t^{\tau}f(x):=\mathbb{E}^{\tau}_x[f(F_t^{\tau})]=\int_{[0,\tau]}f(y)\mathbb{P}_x^{\tau}(F^{\tau}_t\in dy),\quad t\geq0,\ x\in[0,\tau],\ f\in B_b([0,\tau]),$$ where
\begin{equation}\label{eq4.2}\mathbb{P}_x^{\tau}(F^{\tau}_t\in dy):=\sum_{k\in\tau\ZZ^{d}}\mathbb{P}_{z_x}(F_t-k\in dz_y),\quad t>0,\ x,y\in[0,\tau],\end{equation}  and    $z_x$ and $z_y$ are arbitrary points  in
$\Pi^{-1}_{\tau}(\{x\})$ and $\Pi^{-1}_{\tau}(\{y\})$,
respectively.  From \eqref{eq4.2} we automatically conclude that $\process{F^{\tau}}$ has a transition density function $p^{\tau}(t,x,y)$ which is given by
$$p^{\tau}(t,x,y)=\sum_{k\in\tau\ZZ^{d}}p(t,x,y+k),\quad  t>0,\ x,y\in[0,\tau].$$  In particular, if, in addition to the $\tau$-periodicity of the coefficients, $\process{F}$  satisfies  (\textbf{C3}), then
\begin{equation}\label{eq4.3}\inf_{x,y\in[0,\tau]}p^{\tau}(t,x,y)>0,\quad t>0,\end{equation}
which suggests that it is reasonable to expect that  $\process{F^{\tau}}$ is ergodic.

Recall, a probability measure $\pi(dx)$ on a measurable space $(S,\mathcal{S})$ is \emph{invariant} for a
Markov process
$(\Omega,\mathcal{F},\{\mathbb{P}_{x}\}_{x\in
S},\process{\mathcal{F}},\process{\theta},\process{M})$, denoted by $\process{M}$ in the sequel,
 if
$$\int_S\mathbb{P}_{x}(M_t\in B)\pi(dx)=\pi(B),\quad t\geq0,\  B\in\mathcal{S}.$$
  A set $B\in\mathcal{F}$ is said to be
\emph{shift-invariant} if $\theta_t^{-1}B=B$ for all $t\geq0$. The
\emph{shift-invariant} $\sigma$-algebra $\mathcal{I}$ is a
collection of all such shift-invariant sets. The process
$\process{M}$ is said to be \emph{ergodic} if it possesses an invariant probability measure $\pi(dx)$ and if  $\mathcal{I}$ is
trivial with respect to $\mathbb{P}_{\pi}(d\omega)$, that is,
$\mathbb{P}_{\pi}(B)=0$ or $1$ for every $B\in\mathcal{I}$.
Here,
for a probability measure $\mu(dx)$ on $\mathcal{S}$,
$\mathbb{P}_{\mu}(d\omega)$ is defined as
$\mathbb{P}_{\mu}(d\omega):=\int_S\mathbb{P}_{x}(d\omega)\mu(dx).$
The process
$\process{M}$ is said to be \emph{strongly ergodic} if it possesses an invariant probability measure $\pi(dx)$ and if
\begin{equation}\label{eq4.4}\lim_{t\longrightarrow\infty}||\mathbb{P}_{x}(M_t\in\cdot)-\pi(\cdot)||_{TV}=0,\quad  x\in S,\end{equation} where $||\cdot||_{TV}$ denotes the
total variation norm on the space of  signed measures on $\mathcal{S}$. Clearly,  an invariant probability measure of a strongly ergodic Markov processes is unique. In general, the strong ergodicity implies ergodicity (see \cite[Proposition 2.5]{bhat}). On the other hand, ergodicity does not necessarily imply strong ergodicity (for example, see \cite{turbul}).
Sufficient condition for the strong ergodicity of Markov processes is given through the well-known Doeblin's condition.
More precisely, a Markov process $\process{M}$ satisfies \emph{Doeblin's condition} if there exist a probability measure $\mu(dx)$ on $\mathcal{S}$, $t_0>0$ and $\varepsilon\in(0,1)$, such that $\mathbb{P}_x(M_{t_0}\in B)\leq1-\varepsilon$ whenever $\mu(B)\leq\varepsilon$, for all $x\in S$. Moreover, in this case, the convergence in \eqref{eq4.4} is uniformly exponentially fast (see \cite[Theorem VI.2.1]{doob}).
Now, by fixing $t_0>0$ and $$\varepsilon\in\left(0,1\wedge\frac{|\tau|\inf_{x,y\in[0,\tau]}p^{\tau}(t_0,x,y)}{1+\inf_{x,y\in[0,\tau]}p^{\tau}(t_0,x,y)}\right)$$ and taking $\mu(dx):=dx$, where $|\tau|:=\tau_1\cdots\tau_d$,  \eqref{eq4.3} immediately entails that $\process{F^{\tau}}$ satisfies Doeblin's condition. Hence,
$\process{F^{\tau}}$ possesses  a
unique invariant probability measure $\pi(dx)$ satisfying
\begin{equation}\label{eq4.5}\sup\left\{\left|P^{\tau}_t1_{B}(x)-\pi(B)\right|:x\in[0,\tau],\
B\in\mathcal{B}([0,\tau])\right\}\leq Ce^{-ct}\end{equation} for
all $t\geq0$ and some universal constants $c>0$ and $C>0$.
Let us remark that the exponential decay in \eqref{eq4.5} will be crucial in the proof of Theorem \ref{tm3.1}.

\section{Proofs of the Main Results}\label{s5}
Before the proof of Theorem \ref{tm3.1}, let us  recall the notion of characteristics of a semimartingale
(see \cite{jacod}). Let
$(\Omega,\mathcal{F},\process{\mathcal{F}},\mathbb{P},\process{S})$, denoted by
$\process{S}$ in the sequel, be a $d$-dimensional semimatingale and
let $h:\R^{d}\longrightarrow\R^{d}$ be a truncation function (that is, a
continuous and bounded function which satisfies $h(x)=x$ in a neighborhood of
the origin).
 Define,
$$\check{S}(h)_t:=\sum_{s\leq t}(\Delta S_s-h(\Delta S_s)),\ t\geq0,\quad
\textrm{and} \quad S(h)_t:=S_t-\check{S}(h)_t,\ t\geq0,$$ where the process
$\process{\Delta S}$ is defined by $\Delta S_t:=S_t-S_{t-}$ and
$\Delta S_0:=S_0$. The process $\process{S(h)}$ is a \emph{special
semimartingale}, that is, it admits a unique decomposition
\begin{equation}\label{eq:5.1}S(h)_t=S_0+M(h)_t+B(h)_t,\end{equation} where $\process{M(h)}$ is a local
martingale and $\process{B(h)}$ is a predictable process of bounded
variation.
\begin{definition}
 Let $\process{S}$  be a
semimartingale and let $h:\R^{d}\longrightarrow\R^{d}$ be a truncation
function. Furthermore, let $\process{B(h)}$  be the predictable
process
 of bounded variation appearing in (\ref{eq:5.1}),  let $N(\omega,ds,dy)$ be the
compensator of the jump measure
$$\mu(\omega,ds,dy):=\sum_{s:\Delta S_s(\omega)\neq 0}\delta_{(s,\Delta S_s(\omega))}(ds,dy)$$ of the process
$\process{S}$ and let $\process{C}=\{(C_t^{ij})_{1\leq i,j\leq d})\}_{t\geq0}$ be the quadratic co-variation
process for $\process{S^{c}}$ (continuous martingale part of
$\process{S}$), that is,
$C^{ij}_t=\langle S^{i,c}_t,S^{j,c}_t\rangle.$  Then $(B,C,N)$ is called
the \emph{characteristics} of the semimartingale $\process{S}$
 (relative to $h(x)$). In addition, by defining $\tilde{C}(h)^{ij}_t:=\langle
 M(h)^{i}_t,M(h)^{j}_t\rangle$, $i,j=1,\ldots,d$, where $\process{M(h)}$ is the local martingale
 appearing in (\ref{eq:5.1}),  $(B,\tilde{C},N)$ is called the \emph{modified
 characteristics} of the semimartingale $\process{S}$ (relative to $h(x)$).
\end{definition}

Now, we prove Theorem \ref{tm3.1}. We follow the idea from \cite[Theorem 1]{franke-periodic} (see also \cite[Theorem 1.2]{sandric-periodic}).

\begin{proof}[Proof of Theorem \ref{tm3.1}]
First,  recall that, for any initial distribution $\mu(dx)$ of $\process{F}$,
 $\{\varepsilon F_{\varepsilon^{-2}t}\}_{t\geq0}$, $\varepsilon>0$, are
 $\mathbb{P}_{\mu}$- semimartingales (with respect to the natural filtration generated by $\process{F}$) whose (modified) characteristics (relative to $h(x)$) are given by
  \begin{align*}B(h)^{\varepsilon,i}_t&=\frac{1}{\varepsilon^{2}}\int_0^{t}\int_{\R^{d}}\left(h_i(\varepsilon y)-\varepsilon y_i1_{B(0,1)}(y)\right)\nu(F_{\varepsilon^{-2}s},dy)ds,\\
C^{\varepsilon,ij}_t&=\int_0^{t}c_{ij}(F_{\varepsilon^{-2}s})ds,\\
N^{\varepsilon}(ds,B)&=\frac{1}{\varepsilon^{2}}\int_{\R^{d}}1_{B}\left(\varepsilon y\right)\nu(F_{\varepsilon^{-2}s},dy)ds,\\
\tilde{C}(h)^{\varepsilon,ij}_t&=\int_0^{t}c_{ij}(F_{\varepsilon^{-2}s})ds+\frac{1}{\varepsilon^{2}}\int_0^{t}\int_{\R^{d}}h_{i}\left(\varepsilon y\right)h_{j}\left(\varepsilon y\right)\nu(F_{\varepsilon^{-2}s},dy)ds,
\end{align*} $t\geq0$, $B\in\mathcal{B}(\R^{d})$, $i,j=1,\ldots,d$, (see \cite[Lemma 3.2 and Theorem 3.5]{rene-holder} and \cite[Proposition II.2.17]{jacod}).
Now, according to
\cite[Theorem VIII.2.17]{jacod}, in order to prove the desired convergence it suffices to prove that
\begin{equation}\label{eq5.2}B(h)^{\varepsilon,i}_t\xrightarrow[\varepsilon\searrow0]{L^{2}(\mathbb{P}_{\mu},\Omega)} 0
\end{equation} for all  $t\geq0$ and all $i=1,\ldots,d,$
\begin{equation}\label{eq5.3}\int_0^{t}\int_{\R^{d}}g(y)N^{\varepsilon}(ds,dy)\xrightarrow[\varepsilon\searrow0]{L^{2}(\mathbb{P}_{\mu},\Omega)}0
\end{equation} for all $t\geq0$ and all $g\in C_b(\R^{d})$
vanishing in a neighborhood of the origin,
and\begin{equation}\label{eq5.4}\tilde{C}(h)^{\varepsilon,ij}_t\xrightarrow[\varepsilon\searrow0]{L^{2}(\mathbb{P}_{\mu},\Omega)}t\Sigma^{ij}\end{equation}
for all  $t\geq0$ and all $i,j=1,\ldots,d$, where $\Sigma$ is given in \eqref{eq3.2}.

First, we prove the relation in (\ref{eq5.2}). For all $\varepsilon>0$ small enough, we have $$B(h)^{\varepsilon,i}_t=\frac{1}{\varepsilon}\int_0^{t}\int_{B(0,\varepsilon^{-1}\delta)\setminus B(0,1)}y_i\nu(F_{\varepsilon^{-2}s},dy)ds+\frac{1}{\varepsilon^{2}}\int_0^{t}\int_{B^{c}(0,\varepsilon^{-1}\delta)}h_i(\varepsilon y)\nu(F_{\varepsilon^{-2}s},dy)ds,$$ where $\delta>0$ is such that $h(y)=y$ on $B(0,\delta).$ Now, define
\begin{align*}
U_1^{\varepsilon,i}(x)&:=\frac{1}{\varepsilon}\int_{B(0,\varepsilon^{-1}\delta)\setminus B(0,1)}y_i\left(\nu(x,dy)-\int_{[0,\tau]}\nu(z,dy)\pi(dz)\right),\\
U_2^{\varepsilon,i}(x)&:=\frac{1}{\varepsilon^{2}}\int_{B^{c}(0,\varepsilon^{-1}\delta)}h_i(\varepsilon y)\left(\nu(x,dy)-\int_{[0,\tau]}\nu(z,dy)\pi(dz)\right).\end{align*}
Clearly, for all $\varepsilon>0$ small enough, $U_k^{\varepsilon,i}(x)$, $k=1,2,$ are bounded, $\tau$-periodic and  satisfy $U_k^{\varepsilon,i}(F_t)= U_k^{\varepsilon,i}(F^{\tau}_t)$, $t\geq0$, and $$\int_{[0,\tau]}U_k^{\varepsilon,i}(x)\pi(dx)=0.$$
Thus, by the Markov property, \eqref{eq4.2} and (\ref{eq4.5}),
\begin{align*}\mathbb{E}_{\mu}\left[\left(\int_0^{t}U_1^{\varepsilon,i}(F_{\varepsilon^{-2}s})ds\right)^{2}\right]\nonumber&=\mathbb{E}^{\tau}_{\mu}\left[\left(\int_0^{t}U_1^{\varepsilon,i}(F^{\tau}_{\varepsilon^{-2}s})ds\right)^{2}\right] \nonumber\\&=2\int_0^{t}\int_0^{s}\mathbb{E}^{\tau}_{\mu}\left[U_1^{\varepsilon,i}(F^{\tau}_{\varepsilon^{-2}s})U_1^{\varepsilon,i}(F^{\tau}_{\varepsilon^{-2}u})\right]duds
\nonumber\\&=2\int_0^{t}\int_0^{s}\mathbb{E}^{\tau}_{\mu}\left[P^{\tau}_{\varepsilon^{-2}(s-u)}U_1^{\varepsilon,i}(F^{\tau}_{\varepsilon^{-2}u})U_1^{\varepsilon,i}(F^{\tau}_{\varepsilon^{-2}u})\right]\nonumber\\
&\leq2C||U_1^{\varepsilon,i}||^{2}_{\infty}\int_0^{t}\int_0^{s}e^{-c\varepsilon^{-2}(s-u)}duds\nonumber\\
&\leq\frac{4C\varepsilon^{2}t}{c}||U_1^{\varepsilon,i}||^{2}_{\infty}\nonumber\\&=\frac{4Ct}{c}\sup_{x\in[0,\tau]}\left(\int_{B(0,\varepsilon^{-1}\delta)\setminus B(0,1)}y_i\left(\nu(x,dy)-\int_{[0,\tau]}\nu(z,dy)\pi(dz)\right)\right)^{2}\\&\leq \frac{16Ct}{c}   \sup_{x\in[0,\tau]}\left|\int_{B(0,\varepsilon^{-1}\delta)\setminus B(0,1)}y_i\,\nu(x,dy)\right|^{2}.
\end{align*}
Similarly,
\begin{align*}&\mathbb{E}_{\mu}\left[\left(\int_0^{t}U_2^{\varepsilon,i}(F_{\varepsilon^{-2}s})ds\right)^{2}\right]
\nonumber\\&=2\int_0^{t}\int_0^{s}\mathbb{E}^{\tau}_{\mu}\left[P^{\tau}_{\varepsilon^{-2}(s-u)}U_2^{\varepsilon,i}(F^{\tau}_{\varepsilon^{-2}u})U_2^{\varepsilon,i}(F^{\tau}_{\varepsilon^{-2}u})\right]\nonumber\\
&\leq2C||U_2^{\varepsilon,i}||^{2}_{\infty}\int_0^{t}\int_0^{s}e^{-c\varepsilon^{-2}(s-u)}duds
\nonumber\\&
\leq\frac{4C||h||^{2}_{\infty}t}{c\varepsilon^{2}}\sup_{x\in[0,\tau]}\left(\nu\left(x,B^{c}(0,\varepsilon^{-1}\delta)\right)-\int_{[0,\tau]}\nu\left(z,B^{c}(0,\varepsilon^{-1}\delta)\pi(dz)\right)\right)^{2}\nonumber\\
&\leq\frac{16C||h||^{2}_{\infty}\varepsilon^{2}t}{c\delta^{4}}\sup_{x\in[0,\tau]}\left(\int_{B^{c}(0,\varepsilon^{-1}\delta)}|y|^{2}\nu(x,dy)\right)^{2},
\end{align*} where in the final step we used (\textbf{C5}).
Now, for all $\varepsilon>0$ small enough,  \begin{align*}&\left(\mathbb{E}_{\mu}\left[\left(B(h)^{\varepsilon,i}_t\right)^{2}\right]\right)^{\frac{1}{2}}\\&\leq\left(\mathbb{E}_{\mu}\left[\left(\int_0^{t}U_1^{\varepsilon,i}(F_{\varepsilon^{-2}s})ds\right)^{2}\right]\right)^{\frac{1}{2}}+
\left(\mathbb{E}_{\mu}\left[\left(\frac{t}{\varepsilon}\int_{[0,\tau]}\int_{B(0,\varepsilon^{-1}\delta)\setminus B(0,1)}y_i\nu(z,dy)\pi(dz)\right)^{2}\right]\right)^{\frac{1}{2}}\\ &\ \ \ +\left(\mathbb{E}_{\mu}\left[\left(\int_0^{t}U_2^{\varepsilon,i}(F_{\varepsilon^{-2}s})ds\right)^{2}\right]\right)^{\frac{1}{2}}+\left(\mathbb{E}_{\mu}\left[\left(\frac{t}{\varepsilon^{2}}\int_{[0,\tau]}\int_{B^{c}(0,\varepsilon^{-1}\delta)}h_i(\varepsilon y)\nu(z,dy)\pi(dz)\right)^{2}\right]\right)^{\frac{1}{2}}\end{align*}
\begin{align*}&\leq\frac{4C^{1/2}t^{1/2}}{c^{1/2}}\sup_{x\in[0,\tau]}\left|\int_{B(0,\varepsilon^{-1}\delta)\setminus B(0,1)}y_i\,\nu(x,dy)\right|+\frac{t}{\varepsilon}\int_{[0,\tau]}\int_{B^{c}(0,\varepsilon^{-1}\delta)}y_i\nu(z,dy)\pi(dz)\\
&\ \ \ +\frac{4C^{1/2}||h||_{\infty}\varepsilon t^{1/2}}{c^{1/2}\delta^{2}}\sup_{x\in[0,\tau]}\int_{B^{c}(0,\varepsilon^{-1}\delta)}|y|^{2}\nu(x,dy)+\frac{||h||_{\infty}t}{\delta^{2}}\sup_{x\in[0,\tau]}\int_{B^{c}(0,\varepsilon^{-1}\delta)}|y|^{2}\nu(x,y)\\
 &\leq \frac{4C^{1/2}t^{1/2}}{c^{1/2}}\sup_{x\in[0,\tau]}\left|\int_{B(0,\varepsilon^{-1}\delta)\setminus B(0,1)}y_i\,\nu(x,dy)\right|+
\frac{t}{\delta}\sup_{x\in[0,\tau]}\int_{B^{c}(0,\varepsilon^{-1}\delta)}|y|^{2}\nu(x,dy)\\&\ \ \ +\frac{4C^{1/2}||h||_{\infty}\varepsilon t^{1/2}}{c^{1/2}\delta^{2}}\sup_{x\in[0,\tau]}\int_{B^{c}(0,\varepsilon^{-1}\delta)}|y|^{2}\nu(x,dy) +\frac{||h||_{\infty}t}{\delta^{2}}\sup_{x\in\R^{d}}\int_{B^{c}(0,\varepsilon^{-1}\delta)}|y|^{2}\nu(x,y),
\end{align*}
where in the second step, due to (\textbf{C6}), we used the fact that $$\int_{B^{c}(0,1)}y_i\,\nu(x,dy)=0,\quad x\in[0,\tau].$$
Thus, (\ref{eq5.2}) follows by employing (\textbf{C5}) and (\textbf{C6}).

Next, let us prove the relation in (\ref{eq5.3}).
Fix  $g\in C_b(\R^{d})$ which vanishes on $B(0,\delta)$, for some $\delta>0$.
Further, define $$V^{\varepsilon}(x):=\frac{1}{\varepsilon^{2}}\int_{\R^{d}}g(\varepsilon y)\left(\nu(x,dy)-\int_{[0,\tau]}\nu(z,dy)\pi(dz)\right).$$
Clearly, for any $\varepsilon>0$, $V^{\varepsilon}(x)$ has the same properties as $U_1^{\varepsilon,i}(x)$ and $U_2^{\varepsilon,i}(x)$. Thus, similarly as above, we get
\begin{align*}\mathbb{E}_{\mu}\left[\left(\int_0^{t}V^{\varepsilon}(F_{\varepsilon^{-2}s})ds\right)^{2}\right]&\leq2C||V^{\varepsilon,i}||^{2}_{\infty}\int_0^{t}\int_0^{s}e^{-c\varepsilon^{-2}(s-u)}duds\\&\leq\frac{4C||g||^{2}_{\infty}t}{c\varepsilon^{2}}\sup_{x\in[0,\tau]}\left(\nu\left(x,B^{c}(0,\varepsilon^{-1}\delta)\right)-\int_{[0,\tau]}\nu\left(z,B^{c}(0,\varepsilon^{-1}\delta)\right)\pi(dz)\right)^{2}\\
&\leq\frac{16C||g||^{2}_{\infty}\varepsilon^{2}t}{c\delta^{4}}\sup_{x\in[0,\tau]}\left(\int_{B^{c}(0,\varepsilon^{-1}\delta)}|y|^{2}\nu(x,dy)\right)^{2}.
\end{align*}
Consequently, \begin{align*}&\left(\mathbb{E}_{\mu}\left[\left(\int_0^{t}\int_{\R^{d}}g(y)N^{\varepsilon}(ds,dy)\right)^{2}\right]\right)^{\frac{1}{2}}\\&\leq\left(\mathbb{E}_{\mu}\left[\left(\int_0^{t}V^{\varepsilon}(\bar{F}_{\varepsilon^{-2}s})ds\right)^{2}\right]\right)^{\frac{1}{2}}
+\left(\mathbb{E}_{\mu}\left[\left(\frac{t}{\varepsilon^{2}}\int_{[0,\tau]}\int_{\R^{d}}g(\varepsilon y)\nu(z,dy)\pi(dz)\right)^{2}\right]\right)^{\frac{1}{2}}\\&\leq
\frac{4C^{1/2}||g||_{\infty}\varepsilon t^{1/2}}{c^{1/2}\delta^{2}}\sup_{x\in[0,\tau]}\int_{B^{c}(0,\varepsilon^{-1}\delta)}|y|^{2}\nu(x,dy)
+\frac{||g||_{\infty}t}{\delta^{2}}\int_{B^{c}(0,\varepsilon^{-1}\delta)}|y|^{2}\nu(x,dy),
\end{align*}
 which together with (\textbf{C5}) proves (\ref{eq5.3}).

Finally, let us prove the relation in (\ref{eq5.4}).  We have \begin{align*}\tilde{C}(h)^{\varepsilon,ij}_t=&\int_0^{t}c_{ij}(F_{\varepsilon^{-2}s})ds+\int_0^{t}\int_{B(0,\varepsilon^{-1}\delta)}y_iy_j\nu(F_{\varepsilon^{2}s},dy)ds\\&+
\frac{1}{\varepsilon^{2}}\int_0^{t}\int_{B^{c}(0,\varepsilon^{-1}\delta)}h_i(\varepsilon y)h_j(\varepsilon y)\nu(F_{\varepsilon^{2}s},dy)ds,\end{align*} where $\delta>0$ is such that $h(y)=y$ on $B(0,\delta)$. Now, define
\begin{align*}W_1^{\varepsilon,ij}(x)&:=c_{ij}(x)-\int_{[0,\tau]}c^{\tau}_{ij}(z)\pi(dz),\\
W_2^{\varepsilon,ij}(x)&:=\int_{B(0,\varepsilon^{-1}\delta)}y_iy_j\left(\nu(x,dy)-\int_{[0,\tau]}\nu(z,dy)\pi(dz)\right),\\
W_3^{\varepsilon,ij}(x)&:= \frac{1}{\varepsilon^{2}}\int_{B^{c}(0,\varepsilon^{-1}\delta)}h^{i}(\varepsilon y)h(\varepsilon y)\left(\nu(x,dy)-\int_{[0,\tau]}\nu(z,dy)\pi(dz)\right).\end{align*}
Again, for any $\varepsilon>0$, $W_k^{\varepsilon,ij}(x)$, $k=1,2,3$, are bounded, $\tau$-periodic and  satisfy $W_k^{\varepsilon,ij}(F_t)= W_k^{\varepsilon,ij}(F^{\tau}_t)$, $t\geq0$, and $$\int_{[0,\tau]}W_k^{\varepsilon,ij}(x)\pi(dx)=0.$$
Therefore, by using the same arguments as above, from (\ref{eq4.2}) and (\ref{eq4.5}) we get
\begin{equation}\label{eq5.5}\left(\mathbb{E}_{\mu}\left[\left(\int_0^{t}W_1^{\varepsilon,ij}(F_{\varepsilon^{-2}s})ds\right)^{2}\right]\right)^{\frac{1}{2}}\leq\frac{4 C^{1/2}\varepsilon t^{1/2}}{c^{1/2}}||c_{ij}||_{\infty},
\end{equation}
\begin{align}\label{eq5.6}&\left(\mathbb{E}_{\mu}\left[\left(\int_0^{t}\int_{B(0,\varepsilon^{-1}\delta)}y_iy_j\nu(F_{\varepsilon^{-2}s},dy)ds-t\int_{[0,\tau]}\int_{\R^{d}}y_iy_j\nu(z,dy)\pi(dz)\right)^{2}\right]\right)^{\frac{1}{2}}\nonumber\\
&\leq\left(\mathbb{E}_{\mu}\left[\left(\int_0^{t}W_2^{\varepsilon,ij}(F_{\varepsilon^{-2}s})ds\right)^{2}\right]\right)^{\frac{1}{2}}+\left(\mathbb{E}_{\mu}\left[\left(t\int_{[0,\tau]}\int_{B^{c}(0,\varepsilon^{-1}\delta)}y_iy_j\nu(z,dy)\pi(dz)\right)^{2}\right]\right)^{\frac{1}{2}}\nonumber\\&\leq\frac{2C^{1/2}\varepsilon t^{1/2}}{c^{1/2}}\sup_{x\in[0,\tau]}\left|\int_{B(0,\varepsilon^{-1}\delta)}y_iy_j\left(\nu(x,dy)-\int_{[0,\tau]}\nu(z,dy)\pi(dz)\right)\right| \nonumber\\& \ \ \ +t\left|\int_{[0,\tau]}\int_{B^{c}(0,\varepsilon^{-1}\delta)}y_iy_j\nu(z,dy)\pi(dz)\right|\nonumber\\&\leq
\frac{4C^{1/2}\varepsilon t^{1/2}}{c^{1/2}}\sup_{x\in[0,\tau]}\int_{B(0,\varepsilon^{-1}\delta)}|y|^{2}\nu(x,dy)+ t\sup_{x\in[0,\tau]}\int_{B(0,\varepsilon^{-1}\delta)}|y|^{2}\nu(x,dy)\end{align}
and
\begin{align}\label{eq5.7}&\left(\mathbb{E}_{\mu}\left[\left(\frac{1}{\varepsilon^{2}}\int_0^{t}\int_{B^{c}(0,\varepsilon^{-1}\delta)}h_i(\varepsilon y)h_j(\varepsilon y)\nu(F_{\varepsilon^{2}s},dy)ds\right)^{2}\right]\right)^{\frac{1}{2}}\nonumber\\
&\leq\left(\mathbb{E}_{\mu}\left[\left(\int_0^{t}W_3^{\varepsilon,ij}(F_{\varepsilon^{-2}s})ds\right)^{2}\right]\right)^{\frac{1}{2}}\nonumber\\
&\ \ \ +\left(\mathbb{E}_{\mu}\left[\left(\frac{t}{\varepsilon^{2}}\int_{[0,\tau]}\int_{B^{c}(0,\varepsilon^{-1}\delta)}h_i(\varepsilon y)h_j(\varepsilon y)\nu(z,dy)\pi(dz)\right)^{2}\right]\right)^{\frac{1}{2}}\nonumber\\&\leq\frac{2C^{1/2}||g||_{\infty}t^{1/2}}{c^{1/2}\varepsilon}\sup_{x\in[0,\tau]}\left|\nu\left(x,B^{c}(0,\varepsilon^{-1}\delta)\right)-\int_{[0,\tau]}\nu\left(z,B^{c}(0,\varepsilon^{-1}\delta)\pi(dz)\right)\right|\nonumber\\
&\ \ \ + \frac{||h||_{\infty}t}{\varepsilon^{2}}\int_{[0,\tau]}\nu(z,B^{c}(0,\varepsilon^{-1}\delta))\pi(dz)\nonumber\\
&\leq\frac{4C^{1/2}||g||_{\infty}\varepsilon t^{1/2}}{c^{1/2}\delta^{2}}\sup_{x\in[0,\tau]}\int_{B^{c}(0,\varepsilon^{-1}\delta)}|y|^{2}\nu(x,dy)+\frac{||h||_{\infty}t}{\delta^{2}}\int_{B^{c}(0,\varepsilon^{-1}\delta)}|y|^{2}\nu(x,dy).\end{align}
Finally, by combining (\ref{eq5.5}), (\ref{eq5.6}), (\ref{eq5.7}), (\textbf{C5}) and (\textbf{C6}), we get (\ref{eq5.4}).
\end{proof}

In the L\'evy process case we can generalize Theorem \ref{tm3.1}.

\begin{corollary}\label{c5.2}Let  $\process{F}$ be a $d$-dimensional L\'evy process with  L\'evy triplet $(b,c,\nu(dy))$ satisfying condition
(\textbf{C5}).
Then, for any initial
distribution  of $\process{F}$,
$$\left\{\varepsilon\left(F_{\varepsilon^{-2}t}- \varepsilon^{-2}t\left(b+\int_{B^{c}(0,1)}y\,\nu(dy)\right)\right)\right\}_{t\geq0}\xrightarrow[\varepsilon\searrow0]{d}\process{W},$$
 where
$\process{W}$ is a $d$-dimensional zero-drift Brownian motion starting from the origin and
determined by a covariance matrix  of the form
\begin{equation} \label{eq5.8} \Sigma:=\Bigg(c_{ij}+\int_{\R^{d}}y_iy_j\nu(dy)\Bigg)_{1\leq
i,j\leq d}.
\end{equation}
\end{corollary}
\begin{proof}
First, define $$\bar{F}_t:=F_t-t\left(b+\int_{B^{c}(0,1)}y\,\nu(dy)\right),\quad t\geq0.$$ Clearly, $\process{\bar{F}}$ is  a L\'evy process determined by a  symbol of the form $$\bar{q}(\xi)=q(\xi)+i\left\langle\xi,b+\int_{B^{c}(0,1)}y\,\nu(dy)\right\rangle,$$ where $q(\xi)$ is the symbol of $\process{F}.$
Next, as in the proof of Theorem \ref{tm3.1}, we conclude that for any initial distribution $\mu(dx)$ of $\process{F}$, $\{\varepsilon\bar{F}_{\varepsilon^{-2}t}\}$, $\varepsilon>0$, are $\mathbb{P}_{\mu}$-semimartingales  whose (modified) characteristics (relative to a truncation function $h(x)$) are given by
  \begin{align*}B(h)^{\varepsilon,i}_t&=\frac{1}{\varepsilon^{2}}\int_0^{t}\int_{\R^{d}}\left(h_i(\varepsilon y)-\varepsilon y_i1_{B(0,1)}(y)\right)\nu(dy)ds -\frac{t}{\varepsilon}\int_{B^{c}(0,1)}y_i\nu(dy),\\
C^{\varepsilon,ij}_t&=tc_{ij},\\
N^{\varepsilon}(ds,B)&=\frac{1}{\varepsilon^{2}}\int_{\R^{d}}1_{B}\left(\varepsilon y\right)\nu(dy)ds,\quad B\in\mathcal{B}(\R^{d}),\\
\tilde{C}(h)^{\varepsilon,ij}_t&=tc_{ij}+\frac{t}{\varepsilon^{2}}\int_{\R^{d}}h_{i}\left(\varepsilon y\right)h_{j}\left(\varepsilon y\right)\nu(dy)ds,
\end{align*} $i,j=1,\ldots,d$.
Again, according to
\cite[Theorem VIII.2.17]{jacod}, in order to prove the desired convergence it suffices to prove that
$$B(h)^{\varepsilon,i}_t\xrightarrow{\varepsilon\searrow0}0
$$ for all  $t\geq0$ and all $i=1,\ldots,d,$
$$\int_0^{t}\int_{\R^{d}}g(y)N^{\varepsilon}(ds,dy)\xrightarrow{\varepsilon\searrow0}0
$$ for all $t\geq0$ and all $g\in C_b(\R^{d})$
vanishing in a neighborhood of the origin,
and $$\tilde{C}(h)^{\varepsilon,ij}_t\xrightarrow{\varepsilon\searrow0}\Sigma^{ij}$$
for all  $t\geq0$ and all $i,j=1,\ldots,d$, where $\Sigma$ is given in \eqref{eq5.8}. But, the above relations can be easily verified  by employing ($\textbf{C5}$).
\end{proof}

Now, we prove Theorem \ref{tm3.7}.

\begin{proof}[Proof of Theorem \ref{tm3.7}]
Denote by $\mathbb{D}(\R^{d})$  the space of all $\R^{d}$-valued c\`adl\`ag  functions $\alpha:[0,\infty)\longrightarrow\R^{d}$.
It is well known that $\mathbb{D}(\R^{d})$ admits a metrizable topology, called the Skorokhod topology, for which $\mathbb{D}(\R^{d})$ is Polish space (a complete metrizable space).  For more on the Skorokhod topology and space $\mathbb{D}(\R^{d})$ we refer the readers to
\cite{billingsley-book} and \cite{jacod}. Next, let $O\subseteq\R^{d}$ be an  arbitrary  set. Define a function $\tau_O:\mathbb{D}(\R^{d})\longrightarrow[0,\infty]$ by $$\tau_O(\alpha):=\inf\{t>0:\alpha(t)\notin O\}.$$ If $O$ is  open, according to \cite[Lemma 8]{rus}, the set of  continuity points of $\tau_O(\alpha)$ contains the set
$$\mathcal{C}:=\left\{\alpha\in\mathbb{D}(\R^{d}):\lim_{r\longrightarrow0}\tau_{O^{r+}}(\alpha)=\lim_{r\longrightarrow0}\tau_{O^{r-}}(\alpha)=\tau_O(\alpha)\right\},$$ where
$$O^{r+}:=\{x\in\R^{d}:d(x,O)<r\},\quad O^{r-}:=\{x\in\R^{d}:d(x,O^{c})>r\}$$ and $d(x,y):=|x-y|$, $x,y\in\R^{d}$, denotes the standard Euclidian metric on $\R^{d}$. Assume that $O$ satisfies the assumptions from the statement of the theorem and let us prove that  $\mathcal{C}$
contains the sample paths of $\process{W}$ $\mathbb{P}_x$-a.s., $x\in\R^{d}$. Clearly, in order to prove the assertion, it suffices to prove that
\begin{equation}\label{eq5.9}\mathbb{P}_x(\tau_{O\cup\partial O}(\process{W})=0)=\mathbb{P}_x(\tau_{O^{c}}(\process{W})=0)=1,\quad x\in\partial O.\end{equation}
If $d=1$, the assertion immediately follows from   Blumenthal's $0$-$1$ law (see \cite[Theorem 2.8]{peres}).
Assume that $d\geq2$. Then, according to \cite[Theorem 8.3]{peres},  $$\mathbb{P}_x(\tau_{O}(\process{W})=0)=\mathbb{P}_x(\tau_{{\rm Int}\,O^{c}}(\process{W})=0)=1,\quad x\in\partial O.$$
Clearly,
\begin{align*}&\{\omega\in\Omega:\tau_{O}(\{W_t(\omega)\}_{t\geq0})=0\}\setminus\{\omega\in\Omega:\tau_{O\cup\partial O}(\{W_t(\omega)\}_{t\geq0})=0\}\\&\subseteq\{\omega\in\Omega:\textrm{there exists}\ n\geq1,\ \textrm{such that}\ W_t(\omega)\in\partial O\ \textrm{for all}\ t\in[0,1/n]\}\end{align*}and, similarly,
\begin{align*}&\{\omega\in\Omega:\tau_{{\rm Int}\,O^{c}}(\{W_t(\omega)\}_{t\geq0})=0\}\setminus\{\omega\in\Omega:\tau_{O^{c}}(\{W_t(\omega)\}_{t\geq0})=0\}\\&\subseteq\{\omega\in\Omega:\textrm{there exists}\ n\geq1,\ \textrm{such that}\ W_t(\omega)\in\partial O\ \textrm{for all}\ t\in[0,1/n]\}.\end{align*}
Furthermore, for any $x\in\R^{d}$, \begin{align*}&\mathbb{P}_x\left(\left\{\omega\in\Omega:\textrm{there exists}\ n\geq1,\ \textrm{such that}\ W_t(\omega)\in\partial O\ \textrm{for all}\ t\in[0,1/n]\right\}\right)\\&=\mathbb{P}_x\left(\bigcup_{n\geq1}\left\{\omega\in\Omega:W_t(\omega)\in\partial O\ \textrm{for all}\ t\in[0,1/n]\right\}\right)\\&\leq\sum_{n\geq1}\mathbb{P}_x\left(W_{1/n}\in\partial O\right)\\&=0,\end{align*} where in the final step we employed the assumption  that $\partial O$ has Lebesgue measure zero. Hence, the relation in \eqref{eq5.9}  follows.

Now,  we conclude that the set of continuity points of the function $\mathcal{T}_O:\mathbb{D}(\R^{d})\longrightarrow\mathbb{D}(\R^{d})\times[0,\infty]$, $\mathcal{T}_O(\alpha):=(\alpha,\tau_O(\alpha))$, also contains   the sample paths of $\process{W}$ $\mathbb{P}_x$-a.s., $x\in\R^{d}$.  Consequently, by Theorem \ref{tm3.1} and \cite[Theorem 5.1]{billingsley-book}, for any initial distribution of $\process{F}$,
\begin{equation}\label{eq5.10}\left(\left\{\varepsilon F_{\varepsilon^{-2} t}\right\}_{t\geq0},T^{\varepsilon}\right)=\mathcal{T}_O\left(\left\{\varepsilon F_{\varepsilon^{-2} t}\right\}_{t\geq0}\right)\xrightarrow[\varepsilon\searrow0]{d}\mathcal{T}_O\left(\{W_t\}_{t\geq0}\right)=\left(\{W_t\}_{t\geq0},T^{0}\right),\end{equation}
which together with \cite[Corollary 2.5]{aldous} proves \eqref{eq3.7}.

To prove the relation in \eqref{eq3.8}, define a function $\mathcal{S}_O:\mathbb{D}(\R^{d})\longrightarrow\mathbb{D}(\R^{d})$ by $\mathcal{S}_O(\alpha):=\alpha(\cdot\wedge\tau_O(\alpha))$, and  let us prove that  the set of continuity points of $\mathcal{S}_O(\alpha)$ also contains   the sample paths of $\process{W}$ $\mathbb{P}_x$-a.s., $x\in\R^{d}$.
According to \cite[Proposition VI.1.17]{jacod}, it suffices to show that for every $\omega\in\Omega$ such that $\{W_{t}(\omega)\}_{t\geq0}$ is continuous and every sequence $\{\alpha_n\}_{n\geq1}\subseteq\mathbb{D}(\R^{d})$ converging (in the Skorokhod topology) to $\{W_{t}(\omega)\}_{t\geq0}$,  $\mathcal{S}_O(\alpha_n)\xrightarrow{n\nearrow\infty} \mathcal{S}_O(\{W_t(\omega)\}_{t\geq0})$ locally uniformly (or, equivalently, in the Skorokhod topology).
Fix an arbitrary $\omega\in\Omega$ such that $\{W_{t}(\omega)\}_{t\geq0}$ is continuous.
If $\tau_O(\{W_t(\omega)\}_{t\geq0})=\infty$,  the claim easily follows from the continuity of the function $\tau_O(\alpha)$ at sample paths of $\process{W}$ $\mathbb{P}_x$-a.s., $x\in\R^{d}$. Assume now that $\tau_O(\{W_t(\omega)\}_{t\geq0})=t<\infty$, and let $\{\alpha_n\}_{n\geq1}$ be an arbitrary sequence in $\mathbb{D}(\R^{d})$ converging to $\{W_t(\omega)\}_{t\geq0}$ locally uniformly. Because of the continuity of $\{W_t(\omega)\}_{t\geq0}$, for an arbitrary $\epsilon>0$ there exit $0<\delta_\epsilon<t$ and $n_\epsilon\geq1$, such that $|W_s(\omega)-W_u(\omega)|<\epsilon$ for all $s,u\in[t-\delta_\epsilon,t+\delta_\epsilon]$, $|W_s(\omega)-\alpha_n(s)|<\epsilon$ for all $s\in[t-\delta_\epsilon,t+\delta_\epsilon]$ and all $n\geq n_\epsilon$ and $|\tau_O(\alpha_n)-t|<\delta_\epsilon$ for all $n\geq n_\epsilon$. If $K$ is any compact subset of $[0,t-\delta_\epsilon]$ or $[t+\delta_\epsilon,\infty)$, the claim follows trivially. On the other hand, if $K=[t-\epsilon,t+\epsilon]$,
\begin{align*}&\sup_{s\in K}|\mathcal{S}_O(\alpha_n)-\mathcal{S}_O(\{W_t(\omega)\}_{t\geq0})|\\&=\sup_{s\in K}|\alpha_n(s\wedge\tau_O(\alpha_n))-W_{s\wedge t}(\omega)|\\&\leq\sup_{s\in K}|\alpha_n(s\wedge\tau_O(\alpha_n))-W_{s\wedge \tau_O(\alpha_n)}(\omega)|+ \sup_{s\in K}|W_{s\wedge\tau_O(\alpha_n)}(\omega)-W_{s\wedge t}(\omega)|\\
&\leq\sup_{s\in K}|\alpha_n(s)-W_{s}(\omega)|+\sup_{s,u\in K}|W_{s}(\omega)-W_{u}(\omega)|\\
&\leq2\epsilon.\end{align*}
Finally, similarly as in \eqref{eq5.10},  Theorem \ref{tm3.1} and \cite[Theorem 5.1]{billingsley-book} imply that for any initial distribution of $\process{F}$, $$\left\{\varepsilon F_{\varepsilon^{-2} (t\wedge T^{\varepsilon})}\right\}_{t\geq0}=\mathcal{S}_O\left(\left\{\varepsilon F_{\varepsilon^{-2} t}\right\}_{t\geq0}\right)\xrightarrow[\varepsilon\searrow0]{d}\mathcal{S}_O\left(\{W_t\}_{t\geq0}\right)=\{W_{t\wedge T^{0}}\}_{t\geq0},$$
which is \eqref{eq3.8}.
\end{proof}

\section{Non-Zero Drift Case}\label{s6}

In this section, we discuss the non-zero drift case. As we commented in the first section, in  \cite{benso-lions-book} and  \cite{bhat2} the authors have considered the homogenization of a $d$-dimensional diffusion  $\{\varepsilon F_{\varepsilon^{-2}t}\}_{t\geq0}$  determined by $\tau$-periodic coefficients $b(x)=(b_i(x))_{1\leq i\leq d}$ and $c(x)=(c_{ij})_{1\leq i,j\leq d}$, such that $c(x)$ is symmetric and uniformly elliptic, $b_i\in C^{1}(\R^{d})$ with H\"olderian  derivative  and $c_{ij}\in C^{2}(\R^{d})$  with H\"olderian first derivative  (in particular, these assumptions ensure that $\process{F}$ satisfies (\textbf{C3}) and (\textbf{C4})). More precisely,  in \cite[Lemma 3.4.1]{benso-lions-book} they have first shown that there exist unique $\tau$-periodic $\beta_i\in C^{2}(\R^{d})$, $i=1,\ldots,d$, satisfying $$\mathcal{A}^{\infty}\beta_i(x)=b_i(x)-\bar{b}_i,\quad x\in\R^{d},$$ where $\mathcal{A}^{\infty}$ denotes the Feller generator of $\process{F}$, $\bar{b}_i:=\int_{[0,\tau]}b(x)\pi(dx)$ and $\pi(dx)$ is again the unique invariant probability measure associated to $\process{F^{\tau}}$.  Then, by employing the above representation of the drift function, they have obtained that
$$\left\{\varepsilon F_{\varepsilon^{-2}t}-\varepsilon^{-1}\bar{b}t\right\}_{t\geq0} \xrightarrow[\varepsilon\searrow0]{d}\process{\bar{W}}$$
for any initial distribution of $\process{F}$.
Here, $\bar{b}:=(\bar{b}_i)_{1\leq i\leq d}$ and $\process{\bar{W}}$ is  a $d$-dimensional zero-drift Brownian motion starting from the origin and
determined by the covariance matrix  given in \eqref{eq1.4}.

 On the other hand, if $\process{F}$ is a diffusion with jumps with L\'evy triplet $(b(x),c(x),\nu(x,dy))$ satisfying conditions (\textbf{C3})-(\textbf{C6}). Then, by using the approach through the characteristics of semimartingales, the convergence
$$\left\{\varepsilon F_{\varepsilon^{-2}t}-\varepsilon^{-1}\bar{b}t\right\}_{t\geq0}\xrightarrow[\varepsilon\searrow0]{d}\process{\tilde{W}}$$
reduces to the convergences
\begin{equation}\label{eq6.1}\frac{1}{\varepsilon}\int_0^{t}(b_i(F_{\varepsilon^{-2}s})-\bar{b}_i)ds\xrightarrow[\varepsilon\searrow0]{\mathbb{P_\mu}}0,\quad i=1,\ldots,d,\end{equation}
\eqref{eq5.2}, \eqref{eq5.3} and \eqref{eq5.4},  where $\mu(dx)$ is an arbitrary initial distribution od $\process{F}$ and $\process{\tilde{W}}$
is a $d$-dimensional Brownian motion starting from the origin whose covariance matrix has to be
determined.    Note that $\left\{ F_{t}-\bar{b}t\right\}_{t\geq0}$ is again a diffusion with jumps which satisfies (\textbf{C3})-(\textbf{C6}).
However, the relation in \eqref{eq6.1} never holds, unless $b(x)=\bar{b}$ ($dx$-a.e.). Let us be more precise. Let $\process{P^{\tau}}$ be the semigroup of $\process{F^{\tau}}$ (with respect to $(B_b([0,\tau]),||\cdot||_{\infty})$).  Observe that, since $\pi(dx)$ is an
invariant probability measure for $\process{F^{\tau}}$, $\process{P^{\tau}}$ can be naturally
 (and uniquely) extended to a
positivity preserving contraction semigroup on
$(L^{2}([0,\tau],\pi(dx)),||\cdot||_2),$ which we again denote by $\process{P^{\tau}}$. Indeed, for an arbitrary $f\in L^{2}([0,\tau],\pi(dx))$, we have $$||f||_2^{2}=\int_{[0,\tau]}f^{2}(x)\pi(dx)=\int_{[0,\tau]}P_t^{\tau}f^{2}(x)\pi(dx)\geq\int_{[0,\tau]}\left(P_t^{\tau}f(x)\right)^{2}\pi(dx)=||P_t^{\tau}f||_2^{2}.$$
The positivity of $\process{P^{\tau}}$ is trivially satisfied. Next,
the
infinitesimal generator
$(\mathcal{A}_{\tau}^{2},\mathcal{D}_{\mathcal{A}_{\tau}^{2}})$ of
the semigroup $\process{P^{\tau}}$,  with respect to
$(L^{2}([0,\tau],\pi(dx)),||\cdot||_2)$,  is a linear
operator
$\mathcal{A}_{\tau}^{2}:\mathcal{D}_{\mathcal{A}_{\tau}^{2}}\longrightarrow
L^{2}([0,\tau],\pi(dx))$ defined by
$$\mathcal{A}_{\tau}^{2}f:=
  \lim_{t\longrightarrow0}\frac{P^{\tau}_tf-f}{t},\quad f\in\mathcal{D}_{\mathcal{A}_{\tau}^{2}}:=\left\{f\in L^{2}([0,\tau],\pi(dx)):
\lim_{t\longrightarrow0}\frac{P^{\tau}_t f-f}{t} \
\textrm{exists in}\ ||\cdot||_2\right\}.$$
Now,  \cite[Theorem 2.1]{bhat} states that if $f=\mathcal{A}_\tau^{2}g$ for some $g\in \mathcal{D}_{\mathcal{A}_{\tau}^{2}}$ (note that in this case, due to the stationarity of $\pi(dx)$, $\int_{[0,\tau]}f(x)\pi(dx)=0$), then
for any initial distribution of $\process{F}$,
$$\left\{\varepsilon^{-1}\int_0^{t}f(F_{\varepsilon^{-2}s})ds\right\}_{t\geq0}\xrightarrow[\varepsilon\searrow0]{d}\process{W^{1}},$$ where $\process{W^{1}}$ is a one-dimensional zero-drift Brownian motion starting from the origin and determined by the variance parameter $\sigma^{2}=-2\int_{[0,\tau]}f(x)g(x)\pi(dx).$  In what follows we show that  every $\tau$-periodic $f\in B_b(\R^{d})$, that is, its restriction to $[0,\tau]$, satisfying  $\int_{[0,\tau]}f(x)\pi(dx)=0$ is always contained in the range of  $A_\tau^{2}$ (see also \cite[Remark 2.3.1]{bhat}).
Thus, due to the fact that  $p(t,x,dy)$, $t>0$, $x\in\R^{d},$ and $dy$ are mutually absolutely continuous,  \cite[Proposition 2.4]{bhat} implies that \eqref{eq6.1} holds if, and only if,  $b(x)=\bar{b}$ ($dx$-a.e.).

\begin{proposition}\label{p6.1}
 Let $\process{F}$ be a $d$-dimensional diffusion with jumps with  $B_b$-generator $(\mathcal{A}^{b},\mathcal{D}_{\mathcal{A}^{b}})$, satisfying conditions (\textbf{C3}) and (\textbf{C4}). Then, $$\mathcal{E}:=\{f\in C^{2}(\R^{d}):f(x)\ \textrm{is}\
\tau\textrm{-periodic}\}\subseteq\mathcal{D}_{\mathcal{A}^{b}},\quad\mathcal{E}_{\tau}:=\{f|_{[0,\tau]}:f\in \mathcal{E}\}\subseteq\mathcal{D}_{\mathcal{A}_{\tau}^{2}}$$ and, on $\mathcal{E}_{\tau}$ (that is, $\mathcal{E}$), $\mathcal{A}_{\tau}^{2}f|_{[0,\tau]}=(\mathcal{A}^{b}f)|_{[0,\tau]}.$
 \end{proposition}
 \begin{proof}
 First, we show that
$\mathcal{E}\subseteq\mathcal{D}_{\mathcal{A}^{b}}$ and,
on this class of functions,
$\mathcal{A}^{b}$ has the representation   \eqref{eq1.1}.
Let  $\process{P}$ be the  semigroup of $\process{F}$ and let $\mathcal{L}:C_b^{2}(\R^{d})\longrightarrow B_b(\R^{d})$ be defined by
the relation in (\ref{eq1.1}), where $C^{k}_b(\R^{d})$,
$k\geq0$, denotes the space of $k$ times differentiable functions
such that all derivatives up to order $k$ are bounded. Observe that
actually  $\mathcal{L}:C_b^{2}(\R^{d})\longrightarrow C_b(\R^{d})$ (see
\cite[Remark 4.5]{rene-conserv}).   Further, by \cite[Corollary
3.6]{rene-holder},
$$\mathbb{E}_{x}\left[f(F_t)-\int_0^{t}\mathcal{L}f(F_s)ds\right]=f(x),\quad
x\in\R^{d},\ f\in C^{2}_b(\R^{d}).$$
Consequently, for any $\tau$-periodic  $f\in
C^{2}(\R^{d})$,
\begin{align*}\lim_{t\longrightarrow0}\left|\left|\frac{P_tf-f}{t}-\mathcal{L}f\right|\right|_{\infty}&=\lim_{t\longrightarrow0}\left|\left|\frac{1}{t}\int_0^{t}(P_s\mathcal{L}f-\mathcal{L}f)ds\right|\right|_\infty\\&
\leq\lim_{t\longrightarrow0}\frac{1}{t}\int_0^{t}\sup_{x\in[0,\tau]}|P_s\mathcal{L}f(x)-\mathcal{L}f(x)|ds\\&=0,\end{align*}
where in the second step we used the fact that $x\longmapsto
\mathcal{L}f(x)$ is also $\tau$-periodic and in the final step we  applied \cite[Lemma 4.8.7]{jacobI}.

Finally, we show that $\mathcal{E}_{\tau}\subseteq\mathcal{D}_{\mathcal{A}_{\tau}^{2}}$ and $\mathcal{A}_{\tau}^{2}f|_{[0,\tau]}=(\mathcal{A}^{b}f)|_{[0,\tau]},$ $f\in\mathcal{E}$.  Let $f\in \mathcal{E}$. Then, by using the $\tau$-periodicity and \eqref{eq4.2},
\begin{align*}\lim_{t\longrightarrow0}\left|\left|\frac{P^{\tau}_tf|_{[0,\tau]}-f|_{[0,\tau]}}{t}-\left(\mathcal{A}^{b}f\right)\Big|_{[0,\tau]}\right|\right|_{2}&\leq\lim_{t\longrightarrow0}\left|\left|\frac{P^{\tau}_tf|_{[0,\tau]}-f|_{[0,\tau]}}{t}-\left(\mathcal{A}^{b}f\right)\Big|_{[0,\tau]}\right|\right|_{\infty}\\&=\lim_{t\longrightarrow0}\left|\left|\frac{P_tf-f}{t}-\mathcal{A}^{b}f\right|\right|_{\infty}\\&=0,\end{align*}
which concludes the proof.
\end{proof}
\begin{proposition}\label{p6.2} Let $\process{F}$ be a $d$-dimensional diffusion with jumps with semigroup $\process{P}$, satisfying conditions (\textbf{C3}) and (\textbf{C4}). Then, $\process{P^{\tau}}$ is strongly continuous with respect to $(L^{2}([0,\tau],$ $\pi(dx)),||\cdot||_2)$, that is, $\lim_{t\longrightarrow\infty}||P_t^{\tau}f-f||_{2}=0$ for all $f\in L^{2}([0,\tau],\pi(dx)).$
\end{proposition}
\begin{proof} First, recall that the space $C([0,\tau])$ is dense in $L^{2}([0,\tau],\pi(dx))$ with respect to $||\cdot||_2$ (see \cite[Proposition 7.9]{folland}). Next, according to the Stone-Weierstrass theorem (see \cite[Theorem 4.45]{folland}), $C([0,\tau])$ is the closure of the set $\mathcal{E}_{\tau}$ (with respect to $||\cdot||_{\infty}$). Therefore, $\mathcal{E}_{\tau}$ is dense in $L^{2}([0,\tau],\pi(dx))$ (with respect to $||\cdot||_2$). Further, because of the $\tau$-periodicity and \eqref{eq4.2}, it is easy to see that $\process{P^{\tau}}$ is strongly continuous on $\mathcal{E}_{\tau}$ (even with respect to $||\cdot||_\infty$). Finally, let $f\in L^{2}([0,\tau],\pi(dx))$ and $\varepsilon>0$ be arbitrary and let $f_\varepsilon\in\mathcal{E}_{\tau}$ be such that $||f-f_\varepsilon||_2<\varepsilon$. Then,  \begin{align*}||P_t^{\tau}f-f||_2&\leq||P_t^{\tau}f-P_t^{\tau}f_\varepsilon||_2+||P_t^{\tau}f_\varepsilon-f_\varepsilon||_2+||f_\varepsilon-f||_2\\&\leq
2||f-f_\varepsilon||_2+||P_t^{\tau}f_\varepsilon-f_\varepsilon||_2\\&\leq2\varepsilon+||P_t^{\tau}f_\varepsilon-f_\varepsilon||_2,\end{align*} which proves the assertion.
\end{proof}

\begin{proposition}\label{p6.3} Let $\process{F}$ be a $d$-dimensional diffusion with jumps with semigroup $\process{P}$, satisfying conditions (\textbf{C3}) and (\textbf{C4}). Then,  for any  $\tau$-periodic $f\in B_b(\R^{d})$ satisfying  $\int_{[0,\tau]}f(x)\pi(dx)=0$
there exists $g\in\mathcal{D}_{\mathcal{A}_\tau^{2}}$ such
that $f(x)=\mathcal{A}_\tau^{2} g(x)$.
\end{proposition}
\begin{proof}
First, let $h_1,h_2\in L^{2}([0,\tau],\pi(dx))$
be arbitrary functions satisfying  $\int_{[0,\tau]}h_i(x)\pi(dx)=0$, $i=1,2$. Then,
since $\pi(dx)$ is an
invariant probability measure for $\process{F^{\tau}}$, we conclude
$$
\int_{[0,\tau]}P_t^{\tau}h_1(x)h_2(x)\pi(dx)=
\mathbb{E}^{\tau}_\pi\left[\mathbb{E}^{\tau}_\pi[h_1(F^{\tau}_{2t})|F_t^{\tau}]h_2(F_t^{\tau})\right]=\mathbb{E}^{\tau}_\pi\left[h_1(F^{\tau}_{2t})h_2(F_t^{\tau})\right],\quad t\geq0.$$
Now, by employing this fact, the proof of \cite[Corollary 7.2.7]{ethier} and
(\ref{eq4.5}) imply
\begin{align*}
&\left|\int_{[0,\tau]}P_t^{\tau}h_1(x)h_2(x)\pi(dx)\right|\\ &\leq 2
\left(\sup\left\{|\mathbb{E}_\pi^{\tau}\left[1_{B}|F^{\tau}_t\right]-\pi(B)|:\omega\in\Omega,\
B\in\sigma\{F_u^{\tau}: u\geq 2t\}
\right\}\right)^{1/2}||h_1||_2||h_2||_2\\&\leq2\left(\sup\left\{\left|P_t^{\tau}1_{B}(x)-\pi(B)\right|:x\in[0,\tau],\
B\in\mathcal{B}([0,\tau])\right\}\right)^{1/2}||h_1||_2||h_2||_2\\&\leq
2C^{1/2}||h_1||_2||h_2||_2e^{-ct/2}.\end{align*}
Specially,  for any $s,t\geq0$,
$$||P_t^{\tau}h_1||^{2}_2=\left|\int_{[0,\tau]} \left(P^{\tau}_th_1(x)\right)^{2}\pi(dx)\right|\leq
2C^{1/2}||h_1||_2||P_t^{\tau}h_1||_2e^{-ct/2}\leq2C^{1/2}||h_1||^{2}_2e^{-ct/2}$$
and
$$\left|\int_{[0,\tau]} P_t^{\tau}h_1(x)P^{\tau}_sh_2(x)\pi(dx)\right|\leq
2C^{1/2}||h_1||_2||P_s^{\tau}h_2||_2e^{-ct/2}\leq2^{3/2}C^{3/4}||h_1||_2||h_2||_2e^{-ct/2-cs/4}.$$

Next, observe that for each $t\geq0$, $$\int_0^{t}P^{\tau}_sf
ds\in L^{2}([0,\tau],\pi(dx)).$$ Now, we prove that
$\lim_{t\longrightarrow\infty}\int_0^{t}P^{\tau}_sf ds$ exists
in $L^{2}([0,\tau],\pi(dx))$. Let $\{t_n\}_{n\geq1}$
 be an arbitrary sequence in $[0,\infty)$,  $t_n\nearrow\infty$. Then,
the sequence $\left\{\int_0^{t_n}P^{\tau}_sf ds\right\}_{n\geq1}$ is a
Cauchy sequence in $L^{2}([0,\tau],\pi(dx))$. Indeed, for
arbitrary $n,m\in\N$, $n\leq m$,
\begin{align*}\left|\left|\int_0^{t_m}P^{\tau}_sf ds-\int_0^{t_n}P^{\tau}_sf
ds\right|\right|^{2}_2&=\int_{[0,\tau]}\left(\int_{t_n}^{t_m}P^{\tau}_sf(x)
ds\right)^{2}\pi(dx)\\&=\int_{[0,\tau]}\int_{t_n}^{t_m}\int_{t_n}^{t_m}P^{\tau}_tf(x)P^{\tau}_sf(x)dsdt\,\pi(dx)\\&\leq
\int_{t_n}^{t_m}\int_{t_n}^{t_m}\left|\int_{[0,\tau]} P^{\tau}_tf(x)
P^{\tau}_sf(x)\pi(dx)\right|dsdt\\&\leq2^{3/2}C^{3/4}||f||^{2}_2\int_{t_n}^{t_m}\int_{t_n}^{t_m}e^{-ct/2-cs/4}dsdt.
\end{align*}
Thus, since $L^{2}([0,\tau],\pi(dx))$ is a Hilbert space,
the sequence $\left\{\int_0^{t_n}P^{\tau}_sf ds\right\}_{n\geq1}$
converges to some $g\in L^{2}([0,\tau],\pi(dx)).$ By
completely the same reasoning as above, it is easy to see that the function
$g(x)$ does not depend on the choice of a sequence
$\{t_n\}_{n\geq1}$. Therefore,
$$\lim_{t\longrightarrow\infty}\int_0^{t}P^{\tau}_sf
ds=g$$ in $L^{2}([0,\tau],\pi(dx)).$ Furthermore,
because of the continuity (boundedness) of $\process{P^{\tau}}$, for
any $u\geq0$ we have
\begin{equation}\label{eq6.2}P_u^{\tau}g=\lim_{t\longrightarrow\infty}P_u^{\tau}\left(\int_0^{t}P^{\tau}_sf ds\right)=\lim_{t\longrightarrow\infty}\int_u^{t}P^{\tau}_sf ds.\end{equation}

Finally, we show that $g\in\mathcal{D}_{\mathcal{A}_\tau^{2}}$
and $\mathcal{A}_\tau^{2}(-g)=f.$ We have
$$\lim_{t\longrightarrow0}\left(\frac{P_t^{\tau}(-g)+g}{t}-f\right)=\lim_{t\longrightarrow0}\left(\frac{\int_0^{t}P^{\tau}_sf
ds}{t}-f\right)=\lim_{t\longrightarrow0}\frac{\int_0^{t}(P^{\tau}_sf-f)
ds}{t}=0,$$ where in the first step we used \eqref{eq6.2} and in the final step we employed the strong
continuity property of $\process{P^{\tau}}$  (Proposition \ref{p6.2}).
\end{proof}

\section{Long-Time Behavior of Periodic Diffusions with Small Jumps: Transience, Recurrence and Ergodicity}\label{s7}
 In this section, as one of the applications of  Theorem \ref{tm3.1}, we discuss  transience, recurrence and ergodicity of periodic diffusions with small jumps. Denote by ${\rm Leb}(dx)$ the Lebesgue measure on $\mathcal{B}(\R^{d})$.
Recall that  a c\`adl\`ag  strong
Markov process $(\Omega,\mathcal{F},\{\mathbb{P}_{x}\}_{x\in\R^{d}},\process{\mathcal{F}},\process{\theta},\process{M})$, denoted by $\process{M}$
in the sequel, is called
\begin{enumerate}
  \item [(i)] \emph{Lebesgue-irreducible} if  ${\rm Leb}(B)>0$ implies
$\int_0^{\infty}\mathbb{P}_{x}(M_t\in B)dt>0$ for all $x\in\R^{d}$;
  \item [(ii)] \emph{recurrent} if it is
                      Lebesgue-irreducible and if ${\rm Leb}(B)>0$ implies $\int_{0}^{\infty}\mathbb{P}_{x}(M_t\in B)dt=\infty$ for all
                      $x\in\R^{d}$;
\item [(iii)] \emph{Harris recurrent} if it is Lebesgue-irreducible and if ${\rm Leb}(B)>0$ implies $\mathbb{P}_{x}(\tau_B<\infty)=1$ for all
                      $x\in\R^{d}$, where $\tau_B:=\inf\{t\geq0:M_t\in
                      B\}$;
 \item [(iv)] \emph{transient} if it is Lebesgue-irreducible
                       and if there exists a countable
                      covering of $\R^{d}$ with  sets
$\{B_j\}_{j\in\N}\subseteq\mathcal{B}(\R^{d})$, such that for each
$j\geq1$ there is a finite constant $c_j\geq0$ such that
$\int_0^{\infty}\mathbb{P}_{x}(M_t\in B_j)dt\leq c_j$ holds for all
$x\in\R^{d}$.
\end{enumerate}
It is well known that
every Lebesgue-irreducible Markov process is either transient or recurrent  (see \cite[Theorem 2.3]{tweedie-mproc}). Also, clearly,
every Harris recurrent Markov process is recurrent.  But, in general,
these two properties are not equivalent. They differ on a set of
the Lebesgue measure zero (see \cite[Theorem
2.5]{tweedie-mproc}). However,
 for a diffusion with jumps satisfying
condition (\textbf{C3}) these two
 properties are indeed equivalent (see  \cite[Proposition
 2.1]{sandric-tams}). Obviously, a diffusion with jumps satisfying the conditions in (\textbf{C3}) is always Lebesgue-irreducible.
Now, we recall, for our purposes  more adequate,  characterization of
the transience and recurrence through the
sample-paths of the underlying process. Let $B\in\mathcal{B}(\R^{d})$
be  arbitrary. The \emph{sets of transient} and \emph{recurrent functions} in $\mathbb{D}(\R^{d})$, with respect to $B$, are defined as
\begin{align*}T(B)&:=\{\alpha\in\ \mathbb{D}(\R^{d}):\exists s\geq0 \ \textrm{such that} \ \alpha(t)\not\in B, \ \forall t\geq s
\}\\
R(B)&:=\{\alpha\in\ \mathbb{D}(\R^{d}):\forall n\in\N,\ \exists t\geq n \ \textrm{such that} \ \alpha(t)\in B\}
,\end{align*} respectively. Recall that $\mathbb{D}(\R^{d})$ denotes the space of $\R^{d}$-valued c\`adl\`ag functions endowed with the Skorokhod topology.
It is clear that $T(B)=R(B)^{c}$ and for any open set
$O\subseteq\R^{d}$, due to the right continuity of the elements of $\mathbb{D}(\R^{d})$, $T(O)$ and $R(O)$   are
measurable (with respect to the Borel $\sigma$-algebra generated by
the Skorokhod topology). Assume now that $\process{M}$ is a Lebesgue-irreducible Markov process and that for every  compact
set $K\subseteq\R^{d}$ there exists $t_K>0$ such that \begin{equation}\label{eq7.1}\inf_{x\in
K}\mathbb{P}_{x}(M_{t_K}\in B)
>0\end{equation} holds for all $B\in\mathcal{B}(\R^{d})$ satisfying ${\rm Leb}(B)>0$. Then, \cite[Proposition 2.4]{sandric-periodic} asserts that
$\process{M}$  is transient if,
             and only if, $$\mathbb{P}_{x}(\process{M}\in T(O))=1$$ for all $x\in\R^{d}$ and
             all
             open bounded sets $O\subseteq\R^{d}$, and it is
recurrent
if,
             and only if, $$\mathbb{P}_{x}(\process{M}\in R(O))=1$$ for all $x\in\R^{d}$ and
             all
             open bounded sets $O\subseteq\R^{d}$.
According to \cite[Corollary 3.4]{rene-conserv} and \cite[Corollary 2.2]{strong}, any diffusion with jumps satisfying  (\textbf{C3}) automatically satisfies the relation in \eqref{eq7.1}.
\begin{theorem}\label{tm7.1}
Let  $\process{F}$ be a $d$-dimensional diffusion with jumps  satisfying conditions (\textbf{C3})-(\textbf{C6}) and let
 $\process{W}$ be a  $d$-dimensional zero-drift Brownian motion
determined by the relation  in \eqref{eq3.1}. Then, $\process{F}$ is transient (recurrent) if, and only if, $\process{W}$ is transient (recurrent).
\end{theorem}
\begin{proof}
As we commented above, it suffices to prove that $\mathbb{P}_{x}(\process{F}\in T(O))=1$ for all $x\in\R^{d}$ and
             all
             open bounded sets $O\subseteq\R^{d}$ if, and only if, $\mathbb{P}_{x}(\process{W}\in T(O))=1$ for all $x\in\R^{d}$ and
             all
             open bounded sets $O\subseteq\R^{d}$, and $\mathbb{P}_{x}(\process{F}\in R(O))=1$ for all $x\in\R^{d}$ and
             all
             open bounded sets $O\subseteq\R^{d}$ if, and only if, $\mathbb{P}_{x}(\process{W}\in T(O))=1$ for all $x\in\R^{d}$ and
             all
             open bounded sets $O\subseteq\R^{d}$. To see this, first, by \cite[Lemma 3]{franke-periodic}, we have  $\mathbb{P}_{x}(\process{W}\in\partial T(O))=\mathbb{P}_{x}(\process{W}\in\partial R(O))=0$ for all $x\in\R^{d}$ and all open bounded sets $O\subseteq\R^{d}$. Consequently, due to Theorem \ref{tm3.1} and \cite[Theorem 2.1]{billingsley-book}, for any initial distribution $\mu(dx)$ of $\process{F}$ and any  open bounded set $O\subseteq\R^{d}$, $$\lim_{\varepsilon\longrightarrow0}\mathbb{P}_\mu(\{\varepsilon F_{\varepsilon^{-2}t}\}_{t\geq0}\in T(O))=\mathbb{P}_{0}(\process{W}\in T(O))$$ and $$\lim_{\varepsilon\longrightarrow0}\mathbb{P}_\mu(\{\varepsilon F_{\varepsilon^{-2}t}\}_{t\geq0}\in R(O))=\mathbb{P}_{0}(\process{W}\in R(O)).$$ Next, since the processes $\{\varepsilon F_{\varepsilon^{-2}t}\}_{t\geq0}$, $\varepsilon>0,$ also satisfy (\textbf{C3}), $\mathbb{P}_\mu(\{\varepsilon F_{\varepsilon^{-2}t}\}_{t\geq0}\in T(O))$ and  $\mathbb{P}_\mu(\{\varepsilon F_{\varepsilon^{-2}t}\}_{t\geq0}\in R(O))$ are either $0$ or $1$ for all initial distributions $\mu(dx)$ of $\process{F}$, all $\varepsilon>0$ and all open bounded sets $O\subseteq\R^{d}$. Thus, in order to prove the desired result, it suffices to prove that the functions $$\varepsilon\longmapsto\mathbb{P}_\mu(\{\varepsilon F_{\varepsilon^{-2}t}\}_{t\geq0}\in T(O))\quad \textrm{and}\quad \varepsilon\longmapsto\mathbb{P}_\mu(\{\varepsilon F_{\varepsilon^{-2}t}\}_{t\geq0}\in R(O))$$ are continuous for all  initial distributions $\mu(dx)$ of $\process{F}$ and all  open bounded sets $O\subseteq\R^{d}$. But this fact follows directly from
             \begin{align*}\left\{\omega\in\Omega:\{\varepsilon F_{\varepsilon^{-2}t}(\omega)\}_{t\geq0}\in T(O)\right\}&=\left\{\omega\in\Omega:\{F_{t}(\omega)\}_{t\geq0}\in T(\varepsilon^{-1}O)\right\},\\ \left\{\omega\in\Omega:\{\varepsilon F_{\varepsilon^{-2}t}(\omega)\}_{t\geq0}\in R(O)\right\}&=\left\{\omega\in\Omega:\{F_{t}(\omega)\}_{t\geq0}\in R(\varepsilon^{-1}O)\right\}\end{align*} and the characterization of the transience and recurrence through the sample paths of $\process{F}$. Here, for $a\in\R$ and $B\subseteq\R^{d}$, $aB:=\{ab:b\in B\}.$
\end{proof}

At the end, let us comment the (strong) ergodicity of  periodic diffusions with small jumps.
First, recall   that if $\process{M}$ is
a recurrent Markov process, then it admits a unique (up to constant
multiples) invariant measure (see \cite[Theorem 2.6]{tweedie-mproc}). If the invariant measure is
finite, then it may be normalized to a probability measure.  If
$\process{M}$ is  recurrent with finite invariant measure, then $\process{M}$ is called \emph{positive
recurrent}, otherwise it is called \emph{null  recurrent}. One would expect
that every positive  recurrent process  is (strongly) ergodic,
but in general this is not true (see \cite{meyn-tweedie-II} and \cite{turbul}). In the
case of a Lebesgue-irreducible diffusion with jumps $\process{F}$,
these  properties are equivalent.
 Indeed,
according to \cite[Theorem 6.1]{meyn-tweedie-II} and \cite[Theorem
3.3]{rene-wang-feller} it suffices to show that if
 $\process{F}$ admits an invariant probability measure $\pi(dx)$, then
it is recurrent. Assume that $\process{F}$ is transient. Then there
exists a countable
                      covering of $\R^{d}$ with  sets
$\{B_j\}_{j\in\N}\subseteq\mathcal{B}(\R^{d})$, such that for each
$j\geq1$ there is a finite constant $c_j\geq0$ such that
$\int_0^{\infty}\mathbb{P}_x(F_t\in B_j)dt\leq c_j$ holds for all $x\in\R^{d}$.
Let $t>0$ be arbitrary. Then for each $j\geq1$ we have
\begin{equation*}t\pi(B_j)=\int_0^{t}\int_{\R^{d}}\mathbb{P}_x(F_s\in B_j)\pi(dx)ds\leq c_j.\end{equation*} In particular,  by
letting $t\longrightarrow\infty$ we deduce that $\pi(B_j)=0$ for all
$j\geq1$, which is impossible.

Now, directly from Theorem \ref{tm7.1}, we  conclude that a $d$-dimensional, $d\geq3$,  diffusion with jumps $\process{F}$ satisfying (\textbf{C3})-(\textbf{C6}) is never (strongly) ergodic. Recall that a $d$-dimensional zero-drift Brownian motion is recurrent if, and only if, $d=1,2$ (see \cite[Corollary 37.6]{sato-book}). Moreover, $\process{F}$ is neither (strongly) ergodic in $d=1,2$. Indeed, if this were the case, then $\process{F}$ would possess an invariant probability measure, say $\pi(dx)$.  Consequently, for any $\varepsilon>0$,  $\{\varepsilon F_{\varepsilon^{-2}t}\}_{t\geq0}$  also possesses invariant probability measure which is given by $\pi^{\varepsilon}(dx):=\pi(\varepsilon^{-1}dx)$. To see this, first note that $\process{\varepsilon F_{\varepsilon^{-2}}t}$ is a Markov processes with respect to $\mathbb{P}_x^{\varepsilon}(d\omega):=\mathbb{P}_{\varepsilon^{-1}x}(d\omega),$ $x\in\R^{d}.$
Thus, for all $\varepsilon>0$, all $t\geq0$ and all $B\in\mathcal{B}(\R^{d})$, we have
\begin{equation*}\int_{\R^{d}}\mathbb{P}^{\varepsilon}_{x}(\varepsilon F_{\varepsilon^{-2}t}\in B)\pi^{\varepsilon}(dx)= \int_{\R^{d}}\mathbb{P}_{\varepsilon^{-1} x}( \varepsilon F_{\varepsilon^{-2}t}\in B)\pi(\varepsilon^{-1}dx)
=\pi(\varepsilon^{-1}B)=\pi^{\varepsilon}(B).\end{equation*}
Finally, by fixing an arbitrary $t_0>0$, Theorem \ref{tm3.1} implies that \begin{align*}\mathbb{P}_0(W_{t_0}\in B(0,1))&=\lim_{\varepsilon\longrightarrow0}\mathbb{P}_{\pi}(\varepsilon F_{\varepsilon^{-2}t_0}\in B(0,1))\\&=\lim_{\varepsilon\longrightarrow0}\mathbb{P}^{\varepsilon}_{\pi^{\varepsilon}}(\varepsilon F_{\varepsilon^{-2}t_0}\in B(0,1))
\\&=\lim_{\varepsilon\longrightarrow0}\pi^{\varepsilon}(B(0,1))\\&=1,\end{align*} which leads to a contradiction.

 \section*{Acknowledgement}
 This work was supported by the Croatian Science Foundation under Grant 3526 and NEWFELPRO Programme  under Grant 31.

\bibliographystyle{alpha}
\bibliography{References}

\end{document}